\documentclass[11pt]{amsart}
\usepackage{palatino, mathpazo}
\usepackage{amscd}      % to be able to use "CD" environment
\usepackage{amssymb}
\usepackage{amsmath}
\usepackage{dsfont}
\usepackage{xypic}      % to be able to use "diagram" environment of
\LaTeXdiagrams          % xypic to typeset commutative diagrams
\usepackage[all,v2]{xy}
\xyoption{2cell} \UseAllTwocells \xyoption{frame} \CompileMatrices
\usepackage{latexsym}
\usepackage{epsfig}
\usepackage{amsfonts}
\usepackage{enumerate}

\allowdisplaybreaks[3]

\newtheorem{prop}{Proposition}[section]
\newtheorem{lem}[prop]{Lemma}

\newtheorem{cor}[prop]{Corollary}
\newtheorem{thm}[prop]{Theorem}
\newtheorem{rmk}[prop]{Remark}

\newtheorem{conjecture}[prop]{Conjecture}

\newtheorem{defn}[prop]{Definition}

\newtheorem{numex}[prop]{Example}

\newenvironment{pf}{\begin{trivlist}\item[]{\sc Proof.}}%
            {\nolinebreak $\Box$ \end{trivlist}}

\newcommand{\noprint}[1]{}

\renewcommand{\tilde}{\widetilde}

\newcommand{\ber}{{\mbox{\tiny Ber}}}

\newcommand{\et}{{\mbox{\tiny \'{e}t}}}

\newcommand{\com}{^{\scriptscriptstyle\bullet}}

\newcommand{\XX}{{\mathfrak X}}

\newcommand{\YY}{{\mathfrak Y}}
\newcommand{\UU}{{\mathfrak U}}

\newcommand{\ZZ}{{\mathfrak Z}}

\newcommand{\FF}{{\mathfrak F}}

\newcommand{\MM}{{\mathfrak M}}

\newcommand{\zz}{{\mathbb Z}}

\newcommand{\kk}{{\mathbb K}}

\newcommand{\aaa}{{\mathbb A}}

\newcommand{\nn}{{\mathbb N}}

\renewcommand{\ll}{{\mathbb L}}
\newcommand{\qq}{{\mathbb Q}}
\newcommand{\pp}{{\mathbb P}}
\newcommand{\cc}{{\mathbb C}}

\newcommand{\sS}{{\mathcal S}}

\newcommand{\oO}{{\mathcal O}}

\newcommand{\mM}{{\mathcal M}}

\newcommand{\yY}{{\mathcal Y}}

\newcommand{\aA}{{\mathcal A}}

\newcommand{\spf}{\mbox{spf}}
\DeclareMathOperator{\Hom}{Hom}

\newcommand{\ext}{{\mbox{\tiny Ext}}}
\DeclareMathOperator{\Ext}{Ext}
\DeclareMathOperator{\Var}{Var}
\DeclareMathOperator{\Crit}{Crit}
\DeclareMathOperator{\Fsch}{Fsch}
\DeclareMathOperator{\stft}{stft}
\DeclareMathOperator{\Coh}{Coh}
\DeclareMathOperator{\cone}{cone}

\newcommand{\Proj}{\mathop{\rm\bf Proj}\nolimits}

\newcommand{\spec}{\mathop{\rm Spec}\nolimits}

\numberwithin{equation}{section}
\newcommand{\sss}{\vspace{5pt} \subsubsection*{ }\refstepcounter{equation}{{\bfseries(\theequation)}\ }}

\def\Label{\label}

\title[Thom-Sebastiani theorem]{The Thom-Sebastiani theorem for the Euler characteristic of cyclic $L_\infty$-algebras}
\author{Yunfeng Jiang}
\address{Department of Mathematics\\ University of Kansas\\ 405 Snow Hall 1460 Jayhawk Blvd\\Lawrence KS 66045 USA} 
\email{y.jiang@ku.edu}

\begin{document}
\sloppy \maketitle
\begin{abstract}
Let $L$ be a cyclic $L_\infty$-algebra of dimension $3$ with finite dimensional cohomology only in dimension one and two.  
By transfer theorem there exists a cyclic $L_\infty$-algebra structure on the cohomology $H^*(L)$.
The inner product plus the higher products of the cyclic $L_\infty$-algebra defines a superpotential function 
$f$ on $H^1(L)$.  We associate with an analytic Milnor fiber for the formal function $f$ and define the Euler characteristic of $L$ is to be the Euler characteristic of the \'etale cohomology of the analytic Milnor fiber.

In this paper we prove a Thom-Sebastiani type formula for the Euler characteristic of cyclic $L_\infty$-algebras.  As applications we prove the Joyce-Song formulas about the Behrend function identities for semi-Schur objects in the derived category of coherent sheaves over Calabi-Yau threefolds.  A  motivic Thom-Sebastiani type formula and a conjectural motivic Joyce-Song formulas  for the motivic Milnor fiber of cyclic $L_\infty$-algebras are also discussed.
\end{abstract}

%%% -----------------------------------------------------------------------
\maketitle
%%% ----------------------------------------------------------------------

\section{Introduction}

\sss Let $L$ be a cyclic $L_\infty$-algebra of dimension three with finite dimensional cohomology.  The transfer theorem of Kontsevich-Soibelman and Behrend-Getzler  guarantees that on the cohomology $H(L)$ of $L$ there exists a 
cyclic $L_\infty$-algebra structure such that the $L_\infty$-algebra $L$ is quasi-isomorphic to $H(L)$.  In this paper we assume that $H^{i}(L)=0$ only except for $i=1,2$. (This infinity algebra $L$ will control the formal deformation of  the Schur objects in the bounded derived category of coherent sheaves on Calabi-Yau threefolds, hence is related to the moduli problems.)

\sss The cyclic property of the $L_\infty$-algebra structure on $H(L)$ yields a superpotential function $f$ on the first cohomology $X:=H^1(L)$.  The superpotential function $f$ is a formal power series in the coordinates $T_1,\cdots,T_m, S_1,\cdots,S_n$ of $X$.  
Let $R:=\cc[\![t]\!]$ be the complete discrete valuation ring, and $\kk$ its fractional  field. 
Then the  field $\kk$ is $\cc(\!(t)\!)$, which  is a nonarchimedean field with the standard nonarchimedean absolute value 
 $|\cdot|:=e^{-\nu(\cdot)}$, where $\nu(\cdot)$ is the evaluation of $\kk$. The function 
$f$ is naturally an element in the special $R$-algebra $R\{T\}[\![S]\!]:=R\{T_1,\cdots,T_m\}[\![S_1,\cdots,S_n]\!]$.  From \cite{Ber2}, there is a special formal scheme $\spf(R\{T\}[\![S]\!])$ over \spf(R). The generic fiber $\spf(R\{T\}[\![S]\!])_{\eta}$ is the product of a closed unit disc $E^m(0,1)$ and an open unit disc $D(0,1)$ in $\aaa^{m+n}_{\kk,\ber}$. 
The superpotential function $f$ yields a special formal $R$-scheme
$$\hat{f}: \XX:=\spf(A)\to\spf(R)$$
also in the sense of \cite{Ber2}, \cite{Ni}, since $A:=R\{T\}[\![S]\!]/(f-t)$ is a special $R$-algebra. The generic fiber $\XX_\eta$ is a  subanalytic space inside $E^m(0,1)\times D(0,1)$. 
We define an analytic Milnor fiber 
$\FF_{P}(f)$ associated with $f$ using the method in \cite{Ni}, which is an analytic space inside $\XX_\eta$.  We define the Euler characteristic  $\chi(L):=\chi(\FF_{P}(f))$ of $L$ to be the Euler characteristic of the {\'e}tale cohomology of 
$\FF_{P}(f)$ in sense of \cite{Ber1}.

\sss If $f$ is just a regular function in $\cc[T,S]$, then the $R$-algebra $R\{T,S\}$ of strictly convergent power series on $\{T,S\}$ is topologically of finite type over $R$, and the formal scheme $\XX$ is a $\stft$(separated~ and~ topologically~ of ~finite ~type) formal scheme over $R$, and is the $t$-adic completion of the morphism $f: \spec(\cc[T,S])\to \spec(\cc[t])$. The analytic Milnor fiber $\FF_{P}(f)$ associated with $f$ is defined in \cite{NS}, which is the preimage 
$sp^{-1}(P)$ under the specialization  map $sp: \XX_\eta\to\XX_s$. Our definition matches this special case.
Associated to this $f$ there is a topological Milnor fiber $F_P(f)$. By comparison theorem in \cite{Ber4}, \cite{Ber2}, the  Euler characteristic of the topological Milnor fiber $F_P(f)$ is the same as the Euler characteristic of the {\'e}tale cohomology of 
$\FF_{P}(f)$.

\sss Let $X_1$ and $X_2$ be two complex vector spaces and $f_1$ and $f_2$ two holomorphic functions on $X_1$ and $X_2$ respectively.   Then $f=f_1+f_2$ is a holomorphic function on the direct sum space $X=X_1\oplus X_2$.  Let $P=(P_1,P_2)$ be a point in $X$, where $P_i\in X_i$ for $i=1,2$.  The Thom-Sebastiani type formula  for the Milnor number is stated as 
$(1-\chi(\mathbb{F}_{P}))=(1-\chi(\mathbb{F}_{P_{1}}))\cdot (1-\chi(\mathbb{F}_{P_{2}}))$, where 
$\mathbb{F}_{P}$, $\mathbb{F}_{P_{i}}$ are the Milnor fibers of the holomorphic functions $f$ and $f_i$ at $P, P_i$ for 
$i=1,2$.
We prove a Thom-Sebastiani type formula for the Euler characteristic of cyclic $L_\infty$-algebras.  For two  cyclic $L_\infty$-algebras $L_1$ and $L_2$, we construct a cyclic $L_\infty$-algebra 
structure on the direct sum $L:=L_1\oplus L_2$.  The  product formula  $(1-\chi(L))=(1-\chi(L_1))\cdot(1-\cdot\chi(L_2))$
is proved.

\sss Let $Y$ be a smooth projective Calabi-Yau 3-fold, and $D^b(\Coh(Y))$ the bounded derived category of coherent sheaves on $Y$.  The moduli stack of derived objects has not been constructed yet.  In order to apply our method of cyclic $L_\infty$-algebra to the derived objects, we fix a Bridgeland stability condition on $D^b(\Coh(Y))$ and the heart $\mathcal{A}$ of the corresponding bounded $t$-structure is an abelian category.  Any object $E\in \mathcal{A}$ satisfies the condition that $\Ext^{<0}(E,E)=0$. We call such objects semi-Schur. The moduli stack $\MM$ of objects in $\mathcal{A}$ can be constructed, which is an Artin stack locally of finite type.

For an arbitrary object $E$ in the heart $\mathcal{A}$,  we assume that $E$ is Schur or stable under some stability condition, i.e. $\Ext^{i}(E,E)=0$ only except for $i=1,2$.
There is a cyclic dg Lie algebra $R\Hom(E,E)$ corresponding to $E$.  On the cohomology $L_E:=\Ext^*(E,E)$ there is a cyclic $L_\infty$-algebra structure coming from the transfer theorem. 
We define the Euler characteristic $\chi(E)$ of $E$ by the Euler characteristic of the cyclic $L_\infty$-algebra $\Ext^*(E,E)$ or the dg Lie algebra $R\Hom(E,E)$.  Donaldson-Thomas invariants count  stable objects in the derived category and  this Euler characteristic is equal to the pointed Donaldson-Thomas invariant given by the point $E$ in the moduli space.

\sss If $E$ is only semi-Schur, we define $\chi(E)$ to be the Euler characteristic of the \'etale cohomology of the analytic Mionor fiber 
$\FF_{f}(0)$ of the potential function on $\Ext^1(E,E)\to \cc$, where $f$ is just from the cyclic $L_\infty$-algebra structure on 
$\Ext^*(E,E)$.

The work of Behrend and Getzler \cite{BG} proves that  the superpotential function $f: \Ext^1(E,E)\to \cc$ on $\Ext^1(E,E)$  is holomorphic in the complex analytic topology.  The same result for coherent sheaves has been proved in \cite{JS}. 
Let $\MM$ be the moduli stack of objects in $\aA$, which is an Artin stack locally of finite type. 
In \cite{Joyce}, Joyce etc use $(-1)$-shifted symplectic structure of \cite{PTVV} on the moduli  scheme $\MM$ of stable sheaves over smooth Calabi-Yau threefolds to show that the moduli scheme  locally is given by the critical locus of a regular function $g$.  
The Euler characteristic of the topological Milnor fiber associated with the regular function $g$ gives the pointed Donaldson-Thomas invariant. This regular function may not coincide with the superpotential function $f$ coming from the $L_\infty$-algebra at $E$, but they give the same formal germ moduli scheme $\widehat{\MM}_{E}$ at the point $E$. 

\sss We work over nonarchimedean field and the superpotential function $f$ is convergent in an open unit disc of an affine Berkovich analytic space, and the pointed Donaldson-Thomas invariant given by the point $E$ is the Euler characteristic of the analytic Milnor fiber of $f$. 
If $f$ is a regular function, then by the comparison theorem of \cite{Ber4}, the Euler characteristic of the analytic Milnor fiber of $f$ is the same as  the topological  Euler characteristic of the Milnor fiber of $f$. 
If the potential function  $f$ is a holomorphic function in the complex analytic space, then 
 $f$ gives a special formal scheme 
$\hat{f}: \spf(A)\to \spf(R)$. 
A recent preprint \cite{Ber-future} of Berkovich shows that  the nearby and  vanishing cycle functors of $\hat{f}$ defined by him
are isomorphic to the nearby and vanishing functors $\psi_{f}$ and $\varphi_{f}$ in complex analytic topology.  
Thus the comparison theorem in \cite{Ber4} is generalized to the complex analytic topology setting. 
Hence our method gives the same Euler characteristic as in \cite{JS} in the complex analytic topology setting. 

\sss As an application of the Thom-Sebastiani formula, we prove the Joyce-Song formula in \cite{JS} for  semi-Schur objects in the derived category.  The proof is similar to the method of Joyce-Song, except that we use the Euler characteristic of cyclic $L_\infty$-algebras defined in this paper.  This answers Question 5.10 (c) of Joyce in the paper \cite{JS}.  Note that in \cite{Bussi}
V. Bussi uses the $(-1)$-shifted symplectic structure on the moduli stack $\MM$ of coherent sheaves to prove such Behrend function identities, where her proof relies on the local structure of the moduli stack in \cite{Joyce}. 
Our proof is different from theirs and uses Berkovich spaces, and the author hopes that the proof can be generalized to the motivic level of the Behrend function identities, see \cite{Jiang2}. 
The formula is also applied to stable pairs  $\mathcal{O}_{Y}\to F$ constructed in \cite{PT} and proves that the Behrend  function on the stable pair moduli space is, up to a sign,  the Behrend function on the corresponding moduli space of semi-stable sheaves $F$. This result was formerly proved by Bridgeland using the unpublished notes of Pandharipande and Thomas, see Theorem 3.1 of  \cite{Bridgeland}.

\sss The motivic Milnor fiber of the cyclic $L_\infty$-algebras was defined in \cite{Jiang}.  We prove the Thom-Sebastiani formula for the motivic Milnor fibers of cyclic $L_\infty$-algebras.  We also conjecture the motivic version of the Joyce-Song formulas, which is related to a conjecture of Kontsevich-Soibelman about the motivic Milnor fibers. The Kontsevich-Soibelman conjecture was recently proved by Le \cite{Thuong2} using the method of motivic integration. 

\sss The outline of the paper is as follows.  The materials about cyclic $L_\infty$-algebras and transfer theorem is reviewed  in Section \ref{material}.  More details of the theory can be found in \cite{KS}, \cite{BG}.
The germ moduli space of cyclic dg Lie algebras is also discussed. 
In Section \ref{Euler-cyclic} we define the Euler characteristic of the cyclic $L_\infty$-algebra $L$.  We prove a Thom-Sebastiani type formula for Euler characteristic of cyclic $L_\infty$-algebras in 
Section \ref{Thom}.  As applications we prove the Joyce-Song formulas of the Behrend function identities in Section \ref{Joyce}.
Finally in Section \ref{motivic} we talk about the motivic Milnor fiber of cyclic $L_\infty$-algebras and the corresponding Thom-Sebastiani type formula. We also conjecture a motivic version of the Joyce-Song formula for the motivic Milnor fibers. 

\subsection*{Convention}

Although the result in the paper is true for any algebraic closed field $\kappa$ so that the nonarchimedean field is 
$\kappa(\!(t)\!)$ and its ring of integers is $R=\kappa[\![t]\!]$, we work over the complex number $\cc$ throughout the paper. 

For a Berkovich analytic space $\FF$,  we use  $\chi(\FF)$ to represent the Euler characteristics the \'etale cohomology of $\FF$. 
%%% ----------------------------------------------------------------------
\subsection*{Acknowledgments}

The author would like to thank Kai Behrend, Sam Payne and Andrew Strangeway  for valuable discussions on Joyce-Song formula of the Behrend function identities and Berkovich spaces. 
The author is grateful to Y. Toda on the comment of Proposition 5.5 about the Behrend function and the Euler characteristic of the analytic Milnor fiber. 
The project started when the author was visiting KIAS in January 2010, and we thank Kim Bumsig for hospitality.  The author thanks Professor Vladimir Berkovich for email correspondences on vanishing cycle functors, Professor Dominic Joyce for discussion of the Joyce-Song formula for the Behrend function identities, and  Professors Tom Coates, Richard Thomas and Alessio Corti for support at Imperial College London.
This work is partially supported by  Simons Foundation Collaboration Grant 311837.

%%%%%%%%%%%%%%%%%%%%%%%%%%%%%%%%%%%%%%%%%%%%%%%%%%%%%%%%%%%%%%%%%
\section{Cyclic dg Lie Algebra and $L_\infty$-algebra.}\Label{material}
%%%%--------------------------------------------------------------
\subsection*{Cyclic differential graded Lie algebras}
\sss
\begin{defn}\label{dgla}
A \textbf{Differential Graded Lie Algebra} (dg Lie algebra) is a triple $(L,[,],d)$,
where $L=\oplus_{i}L_{i}$ is a $\mathbb{Z}$-graded vector space over $\cc$, the bracket
$[,]: L\times L\longrightarrow L$ is  bilinear and 
$d: L\longrightarrow L$ is a linear map such that 
\begin{enumerate}
\item The bracket $[,]$ is homogeneous skewsymmetric which means that 
$[L^{i},L^{j}]\subset L^{i+j}$ and $[a,b]+(-1)^{\overline{a}\overline{b}}[b,a]=0$
for every $a, b$ homogeneous, where $\overline{a}, \overline{b}$ represent the degrees of $a, b$ respectively.
\item Every homogeneous elements $a, b, c$ satisfy the Jacobi identity:
$$[a,[b,c]]=[[a,b],c]+(-1)^{\overline{a}\overline{b}}[b,[a,c]].$$
\item The map $d$ has degree 1,  $d^{2}=0$ and $d[a,b]=[da,b]+(-1)^{\overline{a}}[a,db]$,
$d$ is called the differential of $L$.
\end{enumerate}
\end{defn}

\sss
\begin{defn}\label{cyclicdgla}
A \textbf{Cyclic Differential Graded Lie Algebra of dimension $3$} is a dg Lie algebra
$L$ together with a bilinear form
$$\ae: L\times L\longrightarrow \cc[-3]$$
such that for any homogeneous elements $a, b, c$,
\begin{enumerate}
\item $\ae$ is graded symmetric, i.e. 
$\ae(a,b)=(-1)^{\overline{a}\overline{b}}\ae(b,a)$;
\item $\ae(da,b)+(-1)^{\overline{a}}\ae(a,db)=0$;
\item $\ae([a,b],c)=\ae(a,[b,c])$;
\item $\ae$ induces a perfect pairing
$$H^{d}(L\otimes L)\longrightarrow \mathbb{C}$$
which induces a perfect pairing 
$$H^{i}(L)\otimes H^{d-i}(L)\longrightarrow \mathbb{C}$$
for any $i$.
\end{enumerate}
\end{defn}
%%%-------------------------------------------------------------
\subsection*{Cyclic $L_\infty$-algebras}

\sss
\begin{defn}\label{l-infty}
An $L_{\infty}$-algebra $L$ is a $\mathbb{Z}$-graded vector space
$\oplus_{i}L^{i}$ equipped with linear maps:
$$\mu_{k}: \Lambda^{k}L\longrightarrow L[2-k],$$
given by 
$a_{1}\otimes\cdots\otimes a_{k}\longmapsto \mu_{k}(a_{1},\cdots,a_{k})$ for $k\geq 1$,
which satisfies the $higher~order~Jacobi~identities$ for any $a_{1},\cdots,a_{n}$:
\begin{align}\label{jacobi-linfty}
\sum_{l=1}^{n}\sum_{\sigma\in Sh(l,n-l)}(-1)^{\widetilde{\sigma}+(n-l+1)(l-1)}
& \epsilon(\sigma;a_{1},\cdots,a_{n})\cdot \\
\nonumber &\mu_{n-l+1}(\mu_{l}(a_{\sigma(1)},\cdots,a_{\sigma(l)}),a_{\sigma(l+1)},\cdots,a_{\sigma(n)})=0,
\end{align}
where the shuffle $Sh(l,n-l)$ is the set of all permutations $\sigma: \{1,\cdots,n\}\rightarrow\{1,\cdots,n\}$
satisfying  $\sigma(1)<\cdots<\sigma(l)$ and  $\sigma(l+1)<\cdots<\sigma(n)$.
The symbol $\epsilon(\sigma;a_{1},\cdots,a_{n})$ (which we abbreviate $\epsilon(\sigma)$) stands for the 
$Koszul~sign$ defined by
$$a_{\sigma(1)}\wedge\cdots\wedge a_{\sigma(n)}=(-1)^{\widetilde{\sigma}}\epsilon(\sigma)a_{1}\wedge\cdots\wedge a_{n},$$
where $\widetilde{\sigma}$ is the parity of the permutation $\sigma$.
\end{defn}

%%%--------------------------------------------------------
\sss \textbf{Decalage}:

Let $L$ be a graded vector space. Let $\odot^{k}L=S^{k}L$, the symmetric algebra
which is defined by tensor algebra modulo the symmetric relations:
$$x_1\odot x_2=(-1)^{x_1x_2}x_2\odot x_1.$$
Let $L[1]$ be the shift to the left by $1$. By definition 
$$L[1]=\mathbb{C}[1]\otimes L,$$
where
$$
\mathbb{C}[1]=\begin{cases}
\mathbb{C}& \mbox{in degree}~ $-1$;\\
0& \mbox{else}.
\end{cases}
$$
Then we have an isomorphism:
$$\phi: \odot^{n}(L[1])\stackrel{\cong}{\longrightarrow}(\Lambda^{n}L)[n]$$
given by
$$(x_1[1]\odot\cdots\odot x_{n}[1])\mapsto (-1)^{\sum_{i=1}^{n}(n-i)x_i}x_1\wedge \cdots\wedge x_{n}[n].$$
Under this isomorphism  we have that
$$q_k:=\mu_k\phi: \odot^{k}(L[1])\longrightarrow L[1]$$
defined by 
$$q_k(x_1[1],\cdots,x_{n}[1])=(-1)^{k+\sum_{i=1}^{k}(k-i)x_i}\mu_k(x_1,\cdots,x_{n})[1]$$
has degree $+1$.
Then the $L_{\infty}$-relation (\ref{jacobi-linfty}) in terms of $(q_{k})$ is:
$$\sum_{l=1}^{n}\sum_{\sigma\in Sh(l,n-l)}
\epsilon(\sigma;a_{1},\cdots,a_{n})q_{n-l+1}(q_{l}(a_{\sigma(1)},\cdots,a_{\sigma(l)}),a_{\sigma(l+1)},\cdots,a_{\sigma(n)})=0.$$

%%%-------------------------------------------------
\sss \textbf{The coalgebra:}

Let $$\bigodot L=\bigoplus_{n\geq 1}\odot^{n}L.$$
We define a coalgebra structure on $\bigodot L$
$$\Delta: \bigodot L\longrightarrow \bigodot L\otimes \bigodot L$$
by
$$x_1\odot\cdots\odot x_n\mapsto 
\sum_{l=1}^{n-1}\sum_{\sigma\in Sh(l,n-l)}\epsilon(\sigma,x)x_{\sigma(1)}\odot\cdots\odot x_{\sigma(l)}\otimes
x_{\sigma(l+1)}\odot\cdots\odot x_{\sigma(n)}.$$
The sequence $(q_k)_{k\geq 1}$  define a coderivation of the coalgebra 
$\bigodot (L[1])$:
$$Q: \bigodot(L[1])\longrightarrow \bigodot(L[1])$$
of degree $1$ given by
$$Q(x_1\odot\cdots\odot x_n)=\sum_{l=1}^{n}
\sum_{\sigma\in Sh(l,n-l)}\epsilon(\sigma,x)q_{l}(x_{\sigma(1)},\cdots,x_{\sigma(l)})\odot 
x_{\sigma(l+1)}\odot\cdots\odot x_{\sigma(n)}.$$
The  $L_{\infty}$-relation (\ref{jacobi-linfty}) is equivalent to 
$Q^2=0$.

\begin{defn}\label{l-infty2}(\textbf{Another definition of $L_{\infty}$-algebras})
An $L_{\infty}$-algebra structure on a $\mathbb{Z}$-graded vector space $L$ is
equivalent to give a codifferential $Q$ on the coalgebra $(S(L[1]),\Delta)$ :
$$Q: S(L[1])\longrightarrow S(L[1])[1]$$
satisfying the conditions
\begin{enumerate}
\item $\Delta\circ Q=(Q\otimes id+id\otimes Q)\circ \Delta$;
\item $Q^{2}=0$.
\end{enumerate}
\end{defn}

\begin{rmk}
We check the first three Jacobi identities from (\ref{jacobi-linfty}). Let $da=\mu_{1}(a)$
and $[a_{1},a_{2}]=\mu_{2}(a_{1},a_{2})$.
\begin{enumerate}
\item (n=1): $d^{2}=0$;
\item (n=2): $d[a_{1},a_{2}]=[da_{1},a_{2}]+(-1)^{\overline{a}_{1}}[a_{1},da_{2}]$;
\item (n=3): $[[a_{1},a_{2}],a_{3}]+(-1)^{(\overline{a}_{1}+\overline{a}_{2})\overline{a}_{3}}[[a_{3},a_{1}],a_{2}]
+(-1)^{\overline{a}_{1}(\overline{a}_{2}+\overline{a}_{3})}[[a_{2},a_{3}],a_{1}]=
-d\mu_{3}(a_{1},a_{2},a_{3})-\mu_{3}(da_{1},a_{2},a_{3})-(-1)^{\overline{a}_{1}}\mu_{3}(a_{1},da_{2},a_{3})
-(-1)^{\overline{a}_{1}+\overline{a}_{2}}\mu_{3}(a_{1},a_{2},da_{3})$.
\end{enumerate}
If $\mu_{k}=0$ for $k\geq 3$, then  the $L_{\infty}$-algebra $L$ is
a dg Lie algebra. One can take the $L_{\infty}$-algebra as generalizations of 
differential graded Lie algebras.
\end{rmk}

\sss 

\begin{defn}\label{linfty-morphism}
An $L_{\infty}$ morphism between two $L_{\infty}$-algebras $L$ and $L^{'}$
is defined by a degree zero coalgebra morphism $F$ from $S(L[1])$ to
$S(L^{'}[1])$ which commutes with the codifferentials $Q$ and $Q^{'}$.
It is completely determined by a set of linear maps $F_{n}: S^{n}(L[1])\rightarrow L^{'}[1-n]$
(or equivalently $F_{n}: \bigwedge^{n}L\rightarrow L^{'}[1-n]$) satisfying a 
set of equations.
\end{defn}

An $L_{\infty}$-morphism $F: L\longrightarrow L^{'}$ is called a $quasi-isomorphism$ if the first
component $F_{1}: L\longrightarrow L^{'}$ induces an isomorphism between cohomology 
groups of complexes $(L,\mu_{1})$ and $(L^{'},\mu_{1}^{'})$. 

Let $L$ be a dg Lie algebra, taking cohomology with respect to the differential $d$ we have 
the cohomological DGLA $H^{\bullet}(L)=\oplus_{i}H^{i}(L)$. We give an $L_{\infty}$-algebra structure 
on $H^{\bullet}(L)$ so that it is quasiisomorphic to $L$ as $L_{\infty}$-algebras.

\sss
\begin{defn}\label{cyclic-l-infty}
A  \textbf{Cyclic $L_{\infty}$-algebra} of dimension $3$ is a triple  $(L,\mu,\ae)$, where 
$(L,\mu)$ is an $L_{\infty}$-algebra with linear maps
$(\mu_{k})$, and 
$$\ae: L\otimes L\longrightarrow \mathbb{C}[-3]$$
is a perfect pairing which means that 
$$H^{i}(L)\otimes H^{d-i}(L)\longrightarrow \mathbb{C}$$
is a perfect pairing of finite dimensional vector spaces  for each $i$. The 
bilinear form $\ae$ satisfies the following conditions:
\begin{enumerate}
\item $\ae$ is graded symmetric, i.e. $\ae(a,b)=(-1)^{ab}\ae(b,a)$;
\item For any $n\geq 1$, 
$$\ae(\mu_{n}(x_1,\cdots,x_n),x_{n+1})=(-1)^{n+x_1(x_2+\cdots+x_{n+1})}\ae(\mu_{n}(x_2,\cdots,x_{n+1}),x_1).$$
\end{enumerate}
\end{defn}

\begin{rmk}
Any cyclic dg Lie algebra in Definition \ref{cyclicdgla} is a cyclic $L_{\infty}$-algebra.
So cyclic $L_{\infty}$-algebras are generalizations of cyclic differential graded Lie
algebras.
\end{rmk}

%%%---------------------------------------------------------------------------
\sss \textbf{Transfer theorem:}
Let $(L,\mu,\ae)$ be a cyclic $L_{\infty}$-algebra, we write $d=\mu_1$. Let
$$\eta: L\longrightarrow L[-1]$$
be map of degree $-1$, such that 
\begin{enumerate}
\item $\eta^{2}=0$;
\item $\eta d\eta=\eta$;
\item $\ae(\eta x,y)+(-1)^{x}\ae(x,\eta y)=0$.
\end{enumerate}
Let $$\Pi=1-[d,\eta],$$
where $[d,\eta]=d\eta+\eta d$. Then $\Pi^{2}=\Pi$. Let
$$H=\Pi(L).$$
Let 
$\iota: H\rightarrow L$ be the inclusion and $p:L\rightarrow H$ the projection. Then
$$\iota p=\Pi, ~~ p\iota=id_{M}.$$

\begin{thm}\label{linfty-coh} (\cite{KS}, \cite{Ka}, \cite{BG})
Let $(L,\mu,\ae)$ be a cyclic $L_{\infty}$-algebra, then there exists a cyclic  $L_{\infty}$-algebra structure on 
$H$ such that there is a $L_{\infty}$-morphism 
$$\varphi: H\longrightarrow L$$
as $L_{\infty}$-algebras.
\end{thm}

%%%-----------------------------------------------
\sss \textbf{A special case:}

Let $(L,\mu,\ae)$ be a cyclic $L_{\infty}$-algebra.
We consider the cohomology of $L$ with respect to the differential 
$d=\mu_1$. Suppose that we have a split:
$$L^{i}=B^i\oplus H^i\oplus K^i,$$
where $B^i$ is the coboundary and $H^i$ the $i$-th cohomology. 
Then we have
$$K^i\stackrel{\cong}{\longrightarrow}B^{i+1}$$
under $d$. Let $q$ be the inverse map. Define the homotopy
$$\eta: L\longrightarrow L[-1]$$
by the following matrix
\begin{equation}\label{eta-special}
\left[\begin{array}{ccc}
0&0&0\\
0&0&0\\
q&0&0
\end{array}\right].
\end{equation}
Then the map $\eta$ satisfies all the properties in the transfer theorem
and $\Pi(L)=H(L)$. So we a
have the following corollary:

\begin{cor}
Let $H=H(L)$, the cohomology $L_{\infty}$-algebra. Then 
$d_{H}=0$ and the $L_{\infty}$-algebra $H(L)$ is 
quasi-isomorphic to $L$ under the morphism $\{\varphi_{i}\}$.
\end{cor}

\begin{pf}
Since on the $L_{\infty}$-algebra $H(L)$, the differential $d_m=0$, it is straightforward 
to check that the $L_{\infty}$-morphism $\{\varphi_i\}$ defined in the above theorem is 
a quasi-isomorphism.
\end{pf}

\subsection*{Moduli space associated to cyclic dg Lie algebras.}

\sss
Let $L$ be a cyclic dg Lie algebra. The germ moduli space of $L$ is given by the Artin stack 
$$\mathfrak{M}=[MC(L)/G],$$
where $MC(L) $ is the Maurer-Cartan space of $L$, and 
$G=\exp(L^0)$ is called the gauge group.

Let $M$ be the coarse moduli space of the stack $\mathfrak{M}$. Suppose that there is a symmetric perfect obstruction theory on $M$ and $M$ is proper, then from \cite{Be} the virtual count of $M$ can be given by the integration of the virtual fundamental class 
$$\int_{[M]^{virt}}1.$$
Let $\nu_{M}$ be the Behrend constructible function on the scheme $X$. Then \cite[Theorem 4.18]{Be} tells us that 
$$\chi(M,\nu_{M})=\int_{[M]^{virt}}1.$$

\sss From the transfer theorem there is an $L_\infty$-algebra structure on the cohomology $H(L)$.  
The $L_\infty$-algebra $H(L)$ and $L$ are quasi-isomorphic, which induces isomorphic germ moduli spaces. 
Using the higher product and the non-degenerate bilinear form one can write down a potential function
$$f: X:=H^{1}(L)\to \cc.$$
Hence the critical locus $Z(df)$ of $f$ is the germ moduli space associated to the dg Lie algebra $L$.
Assume that $f$ is holomorphic, then the value of the Behrend function at the origin is given by
$$\nu_{M}(0)=(-1)^{\dim(M)}(1-\chi(\mathbb{F}_{0})),$$
where $\mathbb{F}_{0}$ is the Milnor fiber of the function $f$ at $0\in X$.

%%%--------------------------------------------------------------
%%%-------------------------------------------------------------
\section{The analytic Milnor fiber associated with cyclic $L_\infty$-algebras.}\label{Euler-cyclic}

\subsection*{The analytic Milnor fiber.}
\sss We define the analytic Milnor fiber for a cyclic $L_\infty$-algebra of dimension 3, which is a 
Berkovich space over the nonarchimedean field $\kk=\cc(\!(t)\!)$.

\sss Let $(L,\mu,\kappa)$ be a cyclic $L_\infty$-algebra. 
We assume that the cyclic structure is of degree 3, so that
$H^i(L)=0$, for $i\not=1,2$, and all $H^i(L)$ are finite
dimensional. Let us write $X$ for the linear manifold $H^1(L)$, and
$P$ for its origin.

By transfer theorem as in \cite{KS}, \cite{BG},  on the cohomology $H^*(L)$ there exists a cyclic 
$L_\infty$-algebra structure $(H^*(L),\nu,\ae)$ such that on $H^1(L)$ there is a potential function 
$$f: H^{1}(L)\to \cc$$
defined in the same way as above.  
The potential $$f=\sum_{k} f_{k}: H^1(L)\longrightarrow \cc$$
is defined by
$$f_{k+1}(x_1,\cdots,x_k,x_{k+1})=\kappa\left(\frac{(-1)^{k(k+1)}}{k!}\mu_k(x_1,\cdots,x_k),x_{k+1}\right).$$
The potential function $f$ is a formal power series over 
$\cc[\![T,S]\!]$, where $T_1, \cdots,T_m$ and $S_1,\cdots,S_n$ are the coordinates of $X=H^1(L)$. 
We assume that $f$ is strictly convergent power series over $T_1, \cdots, T_m$, and formal series over 
$S_1,\cdots, S_n$. 

\sss Let $\kk:=\cc(\!(t)\!)$ be the nonarchimedean field with valuation $v$ such that 
$v(t)=1$. The absolute value $|\cdot|=e^{-v(\cdot)}$. Let $R:=\cc[\![t]\!]$ be the ring of integers of 
$\kk$ under the valuation.  The residue field is $\cc=R/(t)$.

\sss Let $\aA=\kk[T_1,\cdots,T_m]$ be the polynomial ring. 
Let $\aaa^m:=\mM(\aA)$ be the space of multiplicative semi-norms $\Vert\cdot\Vert$ on $\aA$. Then it is a Berkovich analytic space with topology $\Vert\cdot \Vert\to \Vert f\Vert$ is continuous for any $f\in \aA$.
The polynomial ring $\aA$ is not a Banach ring, but there is a cover of Banach analytic domains. Let 
$\aA(\kk)$ be the points $x\in\mM(\aA)$ corresponding to the semi-norms $\Vert\cdot\Vert$ satisfying
$\Vert f\Vert=\vert f(x)\vert$ for $f\in \aA$. 
 Let $\kk\{r^{-1}T\}=\kk\{r_1^{-1}T_1,\cdots,r_m^{-1}T_m\}$ be the Banach algebra of formal convergent series, i.e. $\lim_{i\to\infty}\vert a_i\vert r^{i}\to 0$ for $f=\sum_{i}a_iT^{i}$.
 A closed polydisc $E(0,r)$ can be taken as the affinoid space $\mM(\kk\{r^{-1}T\})$.
 As a set 
the closed disc of radius $r$ is given by 
$$E^m(0,r)=\{x\in \aA(\kk)| \vert x\vert\leq r\}.$$
From \cite{Ber}, the affine analytic  space 
$$\aaa^m=\cup_{r\geq 0} E(0,r)$$
 is an infinite union of closed polydiscs, where $E(0,r)=\{x\in \mM(A)| |T_i|\leq r_i\}$. 
From Berkovich's classification theorem, there are four type of Berkovich points in $\aaa^m$ and each point is the limit of a sequence of points 
$\Vert\cdot\Vert_{D_n}$ corresponding to a nested sequence $D_1\supset D_2\supset\cdots$
of balls of positive radius.

\sss Recall that a topological $R$-algebra $A$ is $special$ if $A$ is an adic ring, and for some ideal of definition $\mathbf{a}\subset A$, the quotient $A/\mathbf{a}^n, n\geq 0$ are finitely generated over $R$.
The potential function $f$ for the cyclic $L_\infty$-algebra $L$ is defined on $M=\cc^{m+n}$.  
As before, we take $f$ as an element in the algebra $R\{T\}[\![S]\!]:=R\{T_1,\cdots,T_m\}[\![S_1,\cdots,S_n]\!]$, which is  
a special $R$-algebra and can be understood as strictly convergent power series over $T_1, \cdots, T_m$, and formal series over 
$S_1,\cdots, S_n$.  The coefficients of $f$ all belong to $\cc$, and the valuations on the elements of $\cc$ are trivial. The algebra 
$R\{T\}[\![S]\!]$ is a $special$ $R$-algebra with the ideal of definition $t\cdot R\{T\}[\![S]\!]$. Let $\spf(R\{T\}[\![S]\!])$ be the special formal scheme over $R$.
The generic fiber is 
$$E^m(0,1)\times D^n(0,1),$$
where $E^m(0,1)$ is the closed unit disc and $D^n(0,1)$ is the unit open disc in the affine space 
$\aaa^m$ and $\aaa^n$ respectively.

\sss In general a formal scheme $\XX\to \spf(R)$ is $special$ if it is a locally finite union of affine formal schemes of the form $\spf(A)$, where $A$ is an adic algebra special over $R$. We fix a locally finite covering $\{\XX_i\}_{i\in I}$, where $\XX_i$ are affine subschemes of the form $\spf(A_i)$ and $A_i$ is a special $R$-algebra.  Let $\XX$ be separated. Then for any $i,j\in I$ the intersection $\XX_{ij}=\XX_i\cap\XX_j$ is also an affine subscheme of the same form.  The generic fiber $\XX_{ij,\eta}$ is a closed analytic domain in $\XX_{i,\eta}$, and the canonical morphism $\XX_{ij,\eta}\to\XX_{i,\eta}\times\XX_{j,\eta}$ is a closed immersion.  From \cite{Ber1}, we can glue $\XX_{i,\eta}$ to get a analytic space $\XX_{\eta}$.

\sss So for the formal potential function $f: \cc^{n+m}\to\cc$, we have a special $R$-algebra $A:=R\{T\}[\![S]\!]/(f-t)$, and we consider the special formal $R$-scheme:
$$\hat{f}: \XX:=\spf(A)\to\spf(R).$$
The generic fiber $\XX_\eta$ is a Berkovich analytic space, which is defined as:
$$\XX_\eta=\{(S,T)\in E^m(0,1)\times D^n(0,1)| f(S,T)=t\}.$$
Then $\XX_\eta$ is a subanalytic space inside $E^m(0,1)\times D^n(0,1)$. 
If we let 
$$\aA=A\otimes_{R}\kk=(R\{T_1,\cdots,T_m\}[\![S_1,\cdots,S_n]\!]\otimes_{R}\kk)/(f-t),$$
then from \cite{Ber2} $\XX_\eta$ is the Berkovich space $\mM(\aA)$, which is identified with the set of continuous multiplicative seminorms on $A$ that extends the valuation on $R$. 

\sss \textbf{The specialization map:} 
The special fiber $\mathfrak{X}_s$ is a $R/(t)=\cc$-scheme $\spec(A/(t))$, which is canonically  isomorphic to the fiber of $f$ over $0$.
The specialization  map  as in \cite[\S 2.2]{NS}
$$sp: \XX_{\eta}\to\XX_{s}$$
sends the points in the generic fibre $\XX_\eta$ to the central fibre $\XX_s$.
%Let $x\in \XX_{\eta}$ be a point, which corresponds to a semi-norm 
%$$x: \aA\to \rr_{\geq 0}.$$
%Let $\hH(x)=\aA/\wp_{x}$, where $\wp_x$ is the kernel of $x$. 
%Then the point $x$ gives a character map 
%$$\tilde{\chi}_{x}: \tilde{A}=A/\mathbf{a}\to \tilde{\hH(x)}=\hH(x)/\aA^{\circ\circ},$$
%where $\mathbf{a}$ is the ideal of definition, and 
%$\aA^{\circ\circ}=\{f\in\aA| |f(x)|< 1 ~\mbox{for all}~x\in\mM(\aA)\}$  is the maximal ideal of 
%$\aA$.  Then the kernel of the map $\tilde{\chi}_x$ is defined to be the image of $x$ under 
%$\pi$. 

Let $\yY\subset \XX_s$ be a closed subset, which is given by an ideal $(\tilde{f}_1,\cdots,\tilde{f}_n)$ for 
$f_i\in A$. Then 
$sp^{-1}(\yY)=\{x\in \XX_\eta| |f_i(x)|<1, 1\leq i\leq n\}$ is open in $\mM(\aA)=\XX_\eta$.
This correspondence means that under the reduction map $\pi$, the preimage of a closed subset is open, and similarly the preimage of an open subset is closed.  This is one of the special properties hold for Berkovich analytic spaces. 

\begin{defn}\label{analyticmilnorfiber}(\cite{NS})
Let $\mathcal{Y}\subset \XX_s$ be a closed subscheme, the analytic Milnor fiber $\mathfrak{F}_{\mathcal{Y}}(f)$ of 
$f$ is defined as 
$$\mathfrak{F}_{\mathcal{Y}}(f)=\{x\in \XX_\eta|f_i(x)|<1, \vert x(f)\vert=0\}.$$
\end{defn}

\sss
\begin{defn}\label{euleranalyticmilnorfiber}
For such a cyclic $L_\infty$-algebra $L$, we define the Euler characteristic $\chi(L)$ of 
$L$ to be the Euler characteristic of the \'etale cohomology of the analytic Milnor fiber 
$\mathfrak{F}_0(f)$ in the sense of Berkovich \cite{Ber1}.
\end{defn}

\begin{rmk}
The reason that the superpotential function $f$ is in $A$ is that the higher products $\mu_k$ in the $L_\infty$-algebra structure may force the function $f$ to be strictly convergent over $T_1,\cdots,T_m$ and a formal power series in $S_1,\cdots,S_n$.

For instance, the superpotential function $f$ coming from the $L_\infty$-algebra of some three dimensional local Calabi-Yau threefolds are polynomials, which in most of the cases coincide with the superpotential function of the associated quiver.  Note that for the moduli stack of stable simple complexes over Calabi-Yau threefolds, in \cite{Joyce} Joyce etc prove that the local potential function of a point is given by regular functions. 
\end{rmk}

\begin{numex}
Let $A=\cc[T]$ be a polynomial ring with one variable $T$. Then there is an homemorphism
$$D(0,1)\stackrel{\cong}{\rightarrow}[0,1]$$ between the  closed unit disc 
$D(0,1)$ and the unit integral $[0,1]$ defined by 
$x\mapsto \vert T(x)\vert$.

Let $f=T^2$. Then the analytic Milnor fiber $\mathfrak{F}_0(f)=\{x\in D(0,1)| \vert T(x)\vert=\epsilon\}$
is a point.
\end{numex}

\begin{numex}
Let $\aA=\kk\{T\}$ be the strictly convergent power series ring with one variable $T$. 
Then $\mM(\aA)$ is isomorphic to the unit closed disc.
Let $f=1+\sum_{i\geq 1}a_iT^i$ be a unit power series satisfying $\vert a_i\vert<1$ and 
$\lim_{i\to\infty}\vert a_i\vert=0$. Then for any multiplicative semi-norm $x$, $x(T)=1$. Then 
if we choose $r<1$ for the open polydisc, the analytic Milnor fiber $\mathfrak{F}_0(f)$ will be empty.
\end{numex}

\subsection*{Polynomial potential function.}

\sss In this section we talk about the case when $f$ is a polynomial, i.e. regular function on the scheme $X$.
Let $f$ be a polynomial in $\cc[T]:=\cc[T_1,\cdots,T_m]$. The analytic Milnor fiber can be defined using the technique of formal schemes.  The polynomial $f$ defines a flat morphism from 
$\aaa^m=\spec(\cc[T])\to \spec(\cc[t])$. Let 
$$\hat{f}: \mathfrak{X}\to \spf(R)$$
be the $t$-adic completion of the morphism $f$, where 
$\mathfrak{X}=\spf (A)$ and $A=R\{T\}/(f-t)$. The algebra  $R\{T\}:=R\{T_1,\cdots,T_m\}$ is the algebra of convergent  power series.
The algebra $A$ is said to be $topologically ~finitely ~generated$ over $R$ and a topologically finitely ~generated algebra over $R$ is
 $special$.
The formal scheme $\mathfrak{X}\to \spf(R)$ is a $\stft$(separated and topologically of finite type)
formal scheme, see \cite{NS}.
The special fiber $\mathfrak{X}_s$ is a $R/(t)=\cc$-scheme $\spec(A/(t))$, which is canonically  isomorphic to the fiber of $f$ over $0$.
The generic fiber $\mathfrak{X}_{\eta}=\mM(A\otimes_{\kk}\kk)$ is a Berkovich analytic 
space over the field $\kk$. 

\sss The generic fiber  $\mathfrak{X}_{\eta}$ is a closed subset in the polydisc
$$E(0,1)=\{x\in\mM(\aA)| \vert T_i(x)\vert\leq 1\},$$
which is defined by the equation $\vert x(f-t)\vert=\vert f(x)\vert-\vert t(x)\vert=0$.
From \cite{NS}, the analytic Milnor fiber $\mathfrak{F}_P(f)$ is define by the preimage 
$sp^{-1}(P)$ under the reduction map. In the special fiber $\mathfrak{X}_s$, the origin 
$P$ is defined by the ideal $(T_1,\cdots,T_m)$ inside $\tilde{A}$. Then 
$$sp^{-1}(P)=\{x\in E(0,1)| \vert T_i(x)\vert< 1,\vert f(x)\vert=\vert t(x)\vert\}.$$
It is easy to see that our definition of analytic Milnor fiber in Definition \ref{analyticmilnorfiber}
matches this construction.

\sss \textbf{Vanishing cycles:}

Let $X:=H^{1}(L)$ and $m=\dim(X)$. 
Consider 
$$
\xymatrix{
&~E\rto^{}\dto_{\hat{\pi}}&~\tilde{\cc}^{*}\dto_{\pi}\\
&~X\setminus f^{-1}(0)\dto_{i}\rto^{\hat{f}}&~\cc^{*} \\
f^{-1}(0)\rto^{j}&X&,}
$$
where $j: f^{-1}(0)\to X$, $i: M\setminus f^{-1}(0)\to X$ are the inclusions,
and $\hat{f}=f|_{X\setminus f^{-1}(0)}$ is the restriction. 
The map $\pi: \tilde{\cc}^{*}$ is the universal cover of $\cc^{*}$ and 
$E$ is the pullback.

Let $\cc$ be the constant sheaf on $X$. The nearby cycle of $f$ is defined by
$$\psi_{f}\cc:=j^{*}R(i\circ\hat{\pi})_{*}(i\circ\hat{\pi})^{*}\cc.$$
Let $x\in f^{-1}(0)$. For any $\epsilon<0$,  let 
$B_{\epsilon}(x)$ be the $\epsilon$-ball of $x$ in $M$. Then 
$$\mathbf{H}^i(\psi_{f}\cc)_{x}\cong H^{i}(B_{\epsilon}(x)\cap X\cap f^{-1}(\delta),\cc),$$
is the cohomology of the Milnor fiber,
where $0<\delta<<\epsilon$. 

The Vanishing Cycle $\varphi_{f}$ is defined by the triangle:
$$
\xymatrix{
j^{*}\cc\rto^{}&~\psi_{f}\cc\dlto_{}\\
\varphi_{f}\cc\uto.}
$$
Then for any  $x\in f^{-1}(0)$, 
$$\mathbf{H}^i(\varphi_{f}\cc)_{x}\cong H^{i}(B_{\epsilon}(x)\cap X, B_{\epsilon}(x)\cap M\cap f^{-1}(\delta),\cc),$$
the reduced cohomology of the Milnor fiber. 

\sss Let $\Fsch$ be the category of $\stft$ formal $R$-schemes.
Let $\XX\in \Fsch$ be a formal $R$-scheme.  For $n\geq 1$, denote the scheme 
$(\XX, \oO_{\XX}/t^n\oO_{\XX})$ by $\XX_n$.  
A morphism of formal schemes over $R$, $\phi: \YY\to\XX$ is said to be \'etale if for all 
$n\geq 1$, the induced morphisms of schemes $\phi_n: \YY_n\to\XX_n$ are \'etale.

Let $\phi: \YY\to \XX$ be a morphism of formal schemes. Then it induces the morphism between the generic and central fibres, i.e.
$\phi_\eta: \YY_\eta\to\XX_\eta$ and $\phi_s: \YY_s\to\XX_s$, where
$\phi_\eta$ is a morphism of Berkovich analytic spaces and $\phi_s$ is a morphism of schemes. 
Here are two known results from  \cite{Ber4} which are needed to construct vanishing cycles.

\begin{lem}\label{functor1}
The correspondence $\YY\mapsto\YY_s$ gives an equivalence between the category of formal schemes \'etale over $\XX$ and the category of schemes \'etale over $\XX_s$.
\end{lem}
\begin{rmk}
Similarly the correspondence $\YY\mapsto\YY_\eta$ gives an equivalence between the category of formal schemes \'etale over $\XX$ and the category of Berkovich analytic spaces \'etale over $\XX_\eta$.
\end{rmk}

\begin{lem}\label{key-lemma1}
Let $\phi: \YY\to \XX$ be an \'etale morphism of formal schemes.
Then  $$\phi_{\eta}(\YY_\eta)=sp^{-1}(\phi_s(\YY_s)).$$
\end{lem}
\begin{pf}
This statement is from the following commutative diagram:
$$
\xymatrix{
\YY_\eta\rto^{\phi_\eta}\dto_{sp}&\XX_{\eta}\dto^{sp}\\
\YY_s\rto^{\phi_s}&\XX_s.}
$$
\end{pf}

\sss The \'etale morphim between $\kk$-analytic spaces can be similarly defined. 
Denote by $\XX_{\eta_{\et}}$ the \'etale site of $\XX_\eta$, which is the site induced from the Grothendieck topology of all \'etale morphisms of $\kk$-analytic spaces. 
Let  $\XX^{\sim}_{\eta_{\et}}$ be the category of sheaves of sets on the \'etale site $\XX_{\eta_{\et}}$.

For two Berkovich $\kk$-analytic spaces $\XX_\eta$ and $\YY_\eta$. A morphism $\psi: \YY_\eta\to\XX_\eta$ is called ``quasi-\'etale" if for every point $y\in \YY_\eta$ there exist affinoid domains $V_{\eta,1}, \cdots, V_{\eta,n}\subset \YY_\eta$ such that the union $V_{\eta,1},\cup\cdots\cup V_{\eta,n}$ is a neighbourhood of $y$ and each $V_{\eta,i}$ is identified with an affinoid domain in a $\kk$-analytic space \'etale over $\XX_\eta$.  

A basic fact from \cite{Ber4} is that an \'etale morphism $\phi: \YY\to \XX$ of formal schemes induces a quasi-\'etale morphism $\phi_\eta: \YY_\eta\to\XX_\eta$ over the generic fibres.
Denote by $\XX_{\eta_{q\et}}$ the quasi-\'etale site of $\XX_\eta$, which is the site induced from the Grothendieck topology of all quasi-\'etale morphisms of $\kk$-analytic spaces.
Let $\XX^{\sim}_{\eta_{q\et}}$ be the category of sheaves of sets on the quasi-\'etale site $\XX_{\eta_{q\et}}$. 
There exists a natural morphism of sites
\begin{equation}
\mu:  \XX^{\sim}_{\eta_{q\et}}\to \XX^{\sim}_{\eta_{\et}}, 
\end{equation}
which is understood as the pullback. 

Let $\YY_s\mapsto\YY$ be the functor obtained from inversing the functor in Lemma \ref{functor1}.
Then from Lemma \ref{key-lemma1} and the fact that \'etale morphisms on formal schemes induce quasi-\'etale morphisms on generic fibres, the composition of the functors $\YY_s\mapsto\YY$ and 
$\YY\mapsto\YY_\eta$ gives a morphism of sites
\begin{equation}
\nu: \XX_{\eta,q\et}\to\XX_{s,\et}.
\end{equation}

\sss 
Let 
\begin{equation}
\Theta=\nu_*\mu^*: \XX^{\sim}_{\eta_{\et}}\to\XX^{\sim}_{\eta_{q\et}}\to\XX^{\sim}_{s_{\et}}
\end{equation}
be the functor obtained from composition. 
Let $\kk^s$ denote the separate closure of the extension of the field $\kk$. 
Let $F$ be an \'etale abelian torsion sheaf over $\XX_\eta$.  Let $\XX_{\overline{\eta}}=\XX_\eta\otimes \kk^s$, and $\overline{F}$ the pullback of $F$ to $\XX_{\overline{\eta}}$.
Then define the nearby cycle functor $\Psi_{\eta}$ by 
\begin{defn} The nearby cycle functor is defined as 
$$\Psi_{\eta}(F)=\Theta_{\widehat{\kk^s}}(\overline{F}).$$
The vanishing cycle functor $\Phi_{\eta}$ is defined to be the cone 
$$\cone[F\to\Psi_{\eta}(F)].$$
\end{defn}

Let $x\in \XX$ be a point in the Maurer-Cartan locus and $\qq_l$ be an \'etale abelian sheaf. Then the stalk 
$$R^{i}\Psi_{\eta}(\qq_l)_x\cong H^i_{\et}(\FF_x,\qq_l)$$ 
is isomorphic to the \'etale cohomology $H^i_{\et}(\FF_x,\qq_l)$ of the analytic Milnor fibre $\FF_x$.  The stalk of the vanishing cycle
$$R^{i}\Phi_{\eta}(\qq_l)_x\cong \tilde{H}^i_{\et}(\FF_x,\qq_l)$$  
is isomorphic to the  reduced \'etale cohomology $\tilde{H}^i_{\et}(\FF_x,\qq_l)$ of the analytic Milnor fibre $\FF_x$. 

\sss \textbf{Comparison Theorem:}

 Let $\hat{F}$ be the corresponding \'etale abelian sheaf on $\XX_\eta$.
 
 \begin{prop}(\cite{Ber4})\label{comparison}
 There exists an isomorphism for an \'etale abelian torsion sheaf $F$ over $X_\eta$:
 $$ i^*(R^qj_*(F))\cong R^q(\Psi_\eta(\hat{F})).$$
 \end{prop}
 
 Let $\hat{F}$ be an \'etale abelian constructible sheaf over the analytic Milnor fibre $\FF_x$, and 
$H^q_{\et}(\FF_x, \hat{F})$ the \'etale cohomology of $\FF_x$.
 
 \begin{prop}
 Let $f$ be a regular function and let $F_x$ be the topological Milnor fibre of $f$ at $x$. 
 Suppose that the formal scheme $\MM$ is the $t$-adic completion of the morphism 
 $f: M=\spec(A)\to\spec(\cc[t])$. Then 
$$H^i(F_x,\cc)\cong H^i(\FF_x, \zz_l)\otimes \cc.$$
This isomorphism is compatible with the monodromy action.
 \end{prop}
 \begin{pf}
 First the sheaf $\zz/l^n$ is an torsion \'etale sheaf. From the comparison Proposition \ref{comparison},
 $$H^i(F_x,\zz/l^n)\cong H^i(\FF_x, \zz/l^n).,$$
 for all $i\geq 0$. Hence 
 \begin{align*}
H^i(\FF_x,\zz_l)&=\underleftarrow{\lim}H^i(\FF_x,\zz/l^n)\\
&= \underleftarrow{\lim}H^i(F_x,\zz/l^n)\\
&= H^i(F_x,\zz_l).
\end{align*}
\end{pf}

\begin{rmk}
Let $M:=\mathbb{V}(df)$, then this is the germ or local moduli space determined by the algebra $L$. 
Let $\nu_{M}$ be the Behrend function in \cite[Definition 1.4]{Be}.
Then
$$\nu_{M}(P)=(-1)^{m}\chi(L).$$
\end{rmk}

\subsection*{\textbf{The Joyce-Song blow up formula}.} 

\sss We prove a result as in \cite[Theorem 4.11]{JS} for Behrend function of blow-ups 
in the formal scheme setting. 
Let $\hat{f}: \XX\to\spf(R)$ be a generically smooth special formal scheme over $R$ and 
$\ZZ\subset \XX$ a closed embedded formal subscheme. 
Let 
$$\phi: \tilde{\XX}\to \XX$$
be the formal blow-up of $\XX$ along $\ZZ$. 
For details of  formal blow-up for $special$ formal schemes see
\cite{Ni}. 
Let $y\in \ZZ\cap \Crit(f)$, then 
 $\phi^{-1}(y)=\pp(T_{y}\XX/T_{y}\ZZ)$ is contained in $\Crit(\tilde{f})$.
 
\begin{prop}\label{lemma-key}
Let $\chi(-)$ be the Euler characteristic of the \'etale cohomology of the analytic spaces. Then
$$\int_{w\in \pp(T_y\XX/T_y\ZZ)}\chi(\FF_{w}(\tilde{f}))d\chi= \chi(\FF_{y}(f))+(\dim(\XX)-\dim(\ZZ)-1)\chi(\FF_{y}(f|_{\ZZ})),$$
where $\int_{w\in \pp(T_y\XX/T_y\ZZ)}\chi(\FF_{w}(\tilde{f}))d\chi$ is understood as the weighted Euler characteristic. 
\end{prop}
\begin{pf}
We may prove  the formula for the case of affine formal schemes.  
Let $\XX=\spf(A)$, where $A=R\{T\}[\![S]\!]/(f-t)$.  
Let $\ZZ\subset \XX$ be a subscheme determined by the ideal $I=(T_{l+1},\cdots,T_m)\subset A$, which is 
$t$-open. 
Then 
$\ZZ=\spf(B)$, where $B=R\{T_1,\cdots,T_l\}[\![S_1,\cdots,S_n]\!]/(f-t)$ for $l<m$.  
The formal blow-up 
$$\tilde{\XX}=\lim_{\substack{\rightarrow\\
n\in\nn}}\Proj \left(\oplus_{d=0}^{\infty}I^d\otimes_{R}(R/t^n)\right).$$
The morphism $\phi$ induces the following commutative diagram:
\begin{equation}\label{diagramblowup}
\xymatrix{
~\tilde{\XX}_{\eta}\dto_{sp}\rto^{\phi_{\eta}}&\XX_{\eta}\dto^{sp}\\
~\tilde{\XX}_s\rto^{\phi_s}&\XX_s,}
\end{equation}
where $sp$ is the specialization map from the generic fiber to the special fiber. 
For any $y\in\ZZ_s\subset \XX_s$, 
$$(sp\circ\phi_{\eta})^{-1}(y)=(\phi_s\circ sp)^{-1}(y).$$

\begin{lem}\label{lemma:1}
We have 
$$\chi((\phi_s\circ sp)^{-1}(y))=\int_{w\in \pp(T_y\XX/T_y\ZZ)}\chi(\FF_{w}(\tilde{f}))d\chi.$$
\end{lem}
\begin{pf}
Since $\phi_s^{-1}(y)$ is the projective space $\pp(T_y\XX/T_y\ZZ)$, the result comes from the inclusion-exclusion and trivial fibration relations of the Euler characteristics. 
\end{pf}

\begin{lem}\label{lemma:2}
$$\chi((sp\circ\phi_{\eta})^{-1}(y))=\chi(\FF_{y}(f))+(\dim(\XX)-\dim(\ZZ)-1)\chi(\FF_{y}(f|_{\ZZ})).$$
\end{lem}
\begin{pf}
We prove that 
$$(sp\circ\phi_{\eta})^{-1}(y)=\left(\FF_{y}(f|_{\ZZ})\times \pp_{\kk,\ber}^{m-n-1}\right)\sqcup (\FF_{y}(f)\setminus \FF_{y}(f|_{\ZZ})),$$
where $\pp_{\kk,\ber}^{m-n-1}$ is the Berkovich projective space. 

The formal scheme $\tilde{\XX}$ is covered by affine formal schemes 
$\UU_i$ for $n+1\leq i\leq m$. The affine formal scheme $\UU_i=\spf(A_i)$ and 
$$A_i=C_i/(t-\mbox{torsion})_{C_i},$$
where 
$$C_i=A\left\{\frac{T_i}{T_j}: j\neq i\right\}=A\{\zeta_j: j\neq i\}/(T_i\zeta_j-T_j, j\neq i).$$ 
The 
torsion is given by
$$(t-\mbox{torsion})_{C_i}=\{f\in C_i: t^nf=0 ~\mbox{for some}~ n\in\nn\}.$$
Let $\aA_i=A_i\otimes\kk$.
Let $\UU_{i,\eta}$ be the generic fiber of $\UU_i$, which is a Berkovich space $\mM(\aA_i)$. Then
\[
\UU_{i,\eta}=\left\{
x\in\mM(\aA_i)\Big| 
\begin{array}{l}
\text{$ \vert T_s(x)\vert\leq 1, 1\leq s\leq l; \vert T_j(x)\vert\leq \vert T_i(x)\vert, l+1\leq j\leq m, j\neq i;$}\\
\text{$ \vert S_t(x)\vert< 1, 1\leq t\leq n;$.}
\end{array}\right\}
\]
Gluing $\UU_{i,\eta}$ for $l+1\leq i\leq m$ we get the generic fiber $\tilde{\XX}_{\eta}$. 

From the definition of $\UU_i$ we write down $\UU_{i,\eta}$ as
\[
\UU_{i,\eta}= 
\left\{
x\in\mM(\aA_i) \Big|
 \begin{array}{l}
  \text{$ \vert T_s(x)\vert\leq 1, 1\leq s\leq l; \vert S_t(x)\vert< 1, 1\leq t\leq n;$} \\
  \text{$ \vert T_j(x)\vert\leq 1$, and $\vert \zeta_j\vert \leq 1,  n+1\leq j\leq m;$} \\
  \text{$\vert \zeta_j(x)\vert\cdot\vert T_i(x)\vert=\vert T_j(x)\vert, n+1\leq j\leq m, j\neq i$.} 
 \end{array}\right\}
\]
Then
\[
\UU_{i,\eta}= 
\left\{
x\in\mM(\aA_i) \Big|
 \begin{array}{l}
  \text{$ \vert T_s(x)\vert\leq 1, 1\leq s\leq l;\vert S_t(x)\vert< 1, 1\leq t\leq n;$} \\
  \text{$\vert \zeta_j\vert \leq 1,  n+1\leq j\leq m;$} \\
  \text{$\vert \zeta_j(x)\vert\cdot\vert T_i(x)\vert=\vert T_j(x)\vert, n+1\leq j\leq m, j\neq i$.}\\
   \text{not all $\vert T_j(x)\vert=0$} 
 \end{array}\right\}
 \]
 \[
 \sqcup
 \left\{
x\in\mM(\aA_i) \Big|
 \begin{array}{l}
  \text{$ \vert T_s(x)\vert\leq 1, 1\leq s\leq l; \vert S_t(x)\vert< 1, 1\leq t\leq n;$} \\
  \text{$\vert \zeta_j\vert \leq 1,  n+1\leq j\leq m;$} \\
  \text{all $\vert T_j(x)\vert=0$ for $j\neq i$} 
 \end{array}\right\}
\] 
\[ 
=\left\{
x\in\mM(\aA_i) \Big|
 \begin{array}{l}
  \text{$ \vert T_s(x)\vert\leq 1, 1\leq s\leq l; \vert S_t(x)\vert< 1, 1\leq t\leq n;$} \\
  \text{$\vert \zeta_j\vert \leq 1,  n+1\leq j\leq m;$} \\
  \text{$\vert \zeta_j(x)\vert\cdot\vert T_i(x)\vert=\vert T_j(x)\vert, n+1\leq j\leq m, j\neq i$.}\\
   \text{not all $\vert T_j(x)\vert=0$} 
 \end{array}\right\}
 \]
 \[
 \sqcup
 \left\{
x\in\mM(\aA_i) \Big|
 \begin{array}{l}
  \text{$ \vert T_s(x)\vert\leq 1, 1\leq s\leq l;$}
 \end{array}\right\}\times E^{m-l-1}(0,1),
\]     
\[
:=\UU_{i,\eta}^{1}\sqcup\UU_{i,\eta}^{2}
\]
where $E^{m-l-1}(0,1)= \{\vert \zeta_j\vert \leq 1,  l+1\leq j\leq m, j\neq i\}$ is the closed unit disc 
in the affine space $\aaa^{m-l-1}$.    
Let $\aA_{\ZZ}=B\otimes_{R}\kk$. 
Then the analytic spaces $\UU_{i,\eta}^2$ glue to form the analytic space
$\{x\in\mM(\aA_{\ZZ})| \vert T_s(x)\vert\leq 1, 1\leq s\leq l; \vert S_t(x)\vert< 1, 1\leq t\leq n;\}\times \pp_{\kk,\ber}^{m-l-1}$, which is 
$\FF_{y}(f|_{\ZZ})\times \pp_{\kk,\ber}^{m-l-1}$.  

The morphism $\phi_{\eta}$ is an isomorphism over the points that not all $\vert T_j(x)\vert=0$ for $l+1\leq j\leq m$. 
The analytic spaces $\UU_{i,\eta}^1$ glue to form the analytic space           
\begin{multline*}
\{x\in\mM(\aA)| \vert T_s(x)\vert\leq 1, 1\leq s\leq m; \vert S_t(x)\vert< 1, 1\leq t\leq n;\}-  \\
 \{x\in\mM(\aA_{\ZZ})| \vert T_s(x)\vert\leq 1, 1\leq s\leq l; \vert S_t(x)\vert< 1, 1\leq t\leq n;\},
\end{multline*} 
which is exactly $(\FF_{y}(f)\setminus \FF_{y}(f|_{\ZZ}))$.        
\end{pf}
From Lemma \ref{lemma:1} and Lemma \ref{lemma:2}, The formula in Proposition \ref{lemma-key} is  proved.
\end{pf}

%%%%%%%%%%%%%%%%%%%%%%%%%%%%%%%
%%%----------------------------------------------------------------------

\section{Thom-Sebastiani Theorem}\label{Thom}

\subsection*{Operation of $L_\infty$-algebras.}

\sss Let  $(L_1, \kappa_1, \mu_1)$ and  $(L_2, \kappa_2, \mu_2)$ be two cyclic $L_\infty$-algebras of dimension $3$.  Define $L:=L_1\oplus L_2$ by
$$L^{n}=L_1^{n}\oplus L_2^{n}.$$
Then 
\begin{equation}
L=\bigoplus_{n\in \zz}L^n=\bigoplus_{n\in \zz}L_{1}^{n}\oplus L_{2}^{n}.
\end{equation}

\begin{defn}For a positive integer $k$, 
let $\mu_k : L^{\otimes k}\to L[2-k]$ be given by 
$$\mu_k=\mu_{1,k}\oplus\mu_{2,k}: (L_1\oplus L_2)^{\otimes k}\to (L_1\oplus L_2)[2-k].$$
\end{defn}

\begin{rmk}
The linear maps $\mu_k$ is given by the diagonal matrix
$$
\mu:=\left[\begin{array}{cc}
\mu_{1,k}&0 \\
0&\mu_{2,k}
\end{array}\right].
$$
Here  $\mu_{1,k}: L_{1}^{\otimes k}\to L_{1}[2-k]$ and $\mu_{2,k}: L_{2}^{\otimes k}\to L_{2}[2-k]$
are the higher linear map of the cyclic $L_\infty$-algebras $L_1$ and $L_2$,  and on the anti-diagonal positions the map is zero.  This means that the map $\mu_k$ on the tensor product 
$L_1\otimes L_{2}$ is zero.
\end{rmk}

\begin{defn}Let 
$$\kappa:=\kappa_1\oplus\kappa_2:  L\otimes L\to \cc$$
be the bilinear form induced from $\kappa_1$ and $\kappa_2$ on 
$L_{1}$ and $L_{2}$. Similar to the above definition, the map $\kappa$ is given by 
$$
\mu:=\left[\begin{array}{cc}
\kappa_1&0 \\
0&\kappa_2
\end{array}\right].
$$
\end{defn}

\begin{thm}
The  higher linear maps $\mu=(\mu_1,\mu_2, \cdots)$ and the bilinear form $\kappa$ defined above give a cyclic $L_\infty$-algebra structure $(L,\mu,\kappa)$ on $L$.
\end{thm}
\begin{pf}
This can be checked using the definition of $L_\infty$-algebras.
\end{pf}

%%%--------------------------------------------------------
\subsection*{The Thom-Sebastiani formula.}
\sss  Let $(L_1, \kappa_1, \mu)$ and  $(L_2, \kappa_2, \nu)$ be two cyclic $L_\infty$-algebras of dimension $3$.  We assume that $H^{i}(L_1)\neq 0$, $H^{i}(L_2)\neq 0$ only for $i=1, 2$.
 Let $X_1=H^{1}(L_1)$,  $X_2=H^{1}(L_2)$ be the linear manifolds, and $P_1$, $P_2$
 the origins.
 
 For simplicity we first assume that $L_i=H^{*}(L_i)$ for $i=1, 2$.
Let $H:=H^{*}(L)$ and $R=S(H[1])$ be the symmetric coalgebra. Let $Q$ be the codifferential.
Then 
$$X:=H^{1}(L)=L^{1}=L_{1}^{1}\oplus L_{2}^{1}=X_1\oplus X_2.$$ 
The potential function 
\begin{equation}
f:  X\to \cc
\end{equation}
is given by the formula
$$
f(z)=\sum_{k=2}^\infty\frac{(-1)^{\frac{k(k+1)}{2}}}{(k+1)!}\kappa\big(\mu_k(z,\ldots,z),z\big).
$$
Since $M=M_1\oplus M_2$, we use the coordinate $z=(z_1,z_2)$, then from the definition of $\mu$ we have that $\mu_{k}(z,\cdots,z)=\mu_{1,k}(z_1,\cdots,z_1)+\mu_{2,k}(z_2,\cdots,z_2)$, then the potential function $f$ satisfies 
$$f(z)=f_{1}(z_1)+f_{2}(z_2),$$
where $f_1$ and $f_2$ are the potential functions of the cyclic $L_\infty$-algebras $L_1$ and $L_{2}$.

\begin{thm}
We have the following formula
$$(1-\chi(L))=(1-\chi(L_{1}))\cdot (1-\chi(L_{2})).$$
\end{thm}
\begin{pf}
This is the classical Thom Sebastiani formula as proved in \cite{DeLo1} and  \cite{Thuong}.
\end{pf}

%%%-----------------------------------------------------------------------

\section{Applications-Derived category of coherent sheaves on Calabi-Yau 3-fold.}\label{Joyce}

\subsection*{Cyclic dg Lie algebra of semi-Schur objects.}
\sss Let $Y$ be a compact Calabi-Yau 3-fold. 
Let $D^{b}(\Coh(Y))$ be the bounded derived category of coherent sheaves on the Calabi-Yau 3-fold 
$Y$.  For an arbitrary derived object $E\in D^{b}(\Coh(Y))$, the formal deformation theory of $E$ is determined by a cyclic $L_\infty$-algebra $\Ext^*(E,E)$ with in general $\Ext^{<0}(E,E)\neq 0$. 
According to \cite{Getzler}, 
this deformation theory is controlled by the corresponding Deligne $l$- or $\infty$-groupoid.  
There exists a construction of the moduli stack of such objects using $\infty$-stacks by B. Toen.  The Behrend function on the moduli stack   has not been defined.  

\sss In the following we fix a Bridgeland stability condition on the derived category $D^{b}(\Coh(Y))$. 
The heart of the corresponding bounded  $t$-structure is an abelian category $\mathcal{A}$. The moduli stack of semi-stable objects in $\mathcal{A}$ can be constructed.  For any object 
$E\in \mathcal{A}$, the statement $\Ext^{<0}(E,E)=0$ is true.  Since $Y$ is Calabi-Yau,   
$\Ext^{>3}(E,E)=0$  holds from Serre duality.  We call such objects "semi-schur".

\sss Let $E\in\mathcal{A}$ be an object and let $R\Hom(E,E)$ be the cyclic dg Lie algebra.  
Choosing an affine open cover $\mathcal{U}$ of $Y$ and there is   a cyclic dg Lie algebra 
$$L(\mathcal{U},E)=\check{C}\com(\mathcal{U},R\Hom(E,E))$$ 
using $\check{C}$ech complex as in \cite{BG},  \cite{Jiang}.  Then on the cohomology 
$\Ext^*(E,E)$ there is a $L_\infty$-algebra structure. 
We denote by $f: \Ext^1(E,E)\to \cc$ the superpotential function.
Let $X_{E}:=\Ext^1(E,E)$.
Then we have the special formal $\cc[\![t]\!]$-scheme 
$$\hat{f}_E: \XX_{E}\to \spf(\cc[\![t]\!])$$
which is the $t$-adic formal completion of $f$. 

\begin{defn}\label{characteristic}
If $E$ is semi-Schur, i.e. $\Ext^{<0}(E,E)=0$, then we define the Euler characteristic  
$$\chi(E)=\chi(\FF_0(\hat{f}_{E})),$$
where $\FF_0(\hat{f}_{E})$ is the analytic Milnor fiber of the origin. 
\end{defn}

\begin{rmk}
If $E$ is Schur or stable, i.e. $\Ext^{<1}(E,E)=0$, then
$$\chi(E)=\chi(L(\mathcal{U}, E))=\chi(\Ext^*(E,E))$$
as in Definition \ref{euleranalyticmilnorfiber}.
\end{rmk}

\begin{lem}
The  Euler characteristic of $E$ does not depend on the choice of the affine
cover $\mathcal{U}$, or the way we write $E$ as a complex of locally frees.
\end{lem}
\begin{pf}
Let $\mathcal{U}_{1}$ and $\mathcal{U}_{2}$ be two open covers of the Calabi-Yau
3-fold $X$, and let 
$$L_1=L(\mathcal{U}_{1},E), ~~~~L_2=L(\mathcal{U}_{2},E).$$
Then from the fact about the Cech cohomology, $L_1$ and $L_2$ are 
quasi-isomorphic, and so $H(L_1)=H(L_2)=\Ext^*(E,E)$ as cohomology.

From the fact of deformation theory
the quasi-isomorphism of $L_1$ and $L_2$ gives the isomorphism germ moduli space 
at $E$, i.e. the germ moduli spaces $(\Ext^1(E,E), f_1)$ and $(\Ext^1(E,E), f_2)$ are isomorphic. 
And then from Definition \ref{characteristic} the Euler characteristic of $L_1$ and $L_2$ are the same.
\end{pf}

\sss Let $\MM$ be such  moduli stack of coherent semi-stable objects in the abelian category $\mathcal{A}$. 
These will contain Schur and semi-Schur objects in the derived category.
Then for any object $E\in \mathcal{A}$, the moduli stack locally can given by
$$[MC(L_{E})/G_{E}],$$
where $L_{E}$ is the associated dg Lie algebra of $E$, $MC(L_E)$ is the Maurer-Carten space and 
$G_E$ is the gauge group. 
Suppose that $E$ is a semi-Schur object, then on the cohomology $\Ext^*(E,E)$, there is a cyclic 
$L_\infty$-algebra structure such that the moduli stack is locally isomorphic to 
$$MC(\Ext^*(E,E)),$$
where the Maurer-Cartan space is given by the superpotential function 
$$f_E: \Ext^1(E,E)\to \cc$$
 coming from the $L_\infty$-algebra structure. 
Let 
$$M:=MC(\Ext^*(E,E))=Z(df).$$
\begin{defn}\label{Kai-function}(\cite{JS})
The Behrend function on $\MM$ is defined as 
$$\nu_{\MM}(E):=(-1)^{\dim(\ext^{0}(E,E))}\nu_{M}(0),$$
where $\nu_{M}$ is the original Behrend function of $M$ defined in \cite[Definition 1.4]{Be}.
\end{defn}

\begin{prop}
We have:
$$\nu_{M}(0)=(-1)^{\dim(\ext^1(E,E))}(1-\chi(E)).$$
\end{prop}
\begin{proof}
From Definition \ref{characteristic}, the  Euler characteristic $\chi(E)$ of the semi-Schur object $E$ is the Euler characteristic of the analytic Milnor fiber $\FF_0(\hat{f}_{E})$. 
We use the result of Behrend-Getzler in \cite{BG} that $f$ is holomorphic.  
From the comparison theorem of Berkovich in \cite{Ber-future}, the vanishing cycle functor of the special formal scheme 
$$\hat{f}_{E}: \XX\to \spf(R)$$
associated to the holomorphic function $f_{E}$ is isomorphic to the vanishing cycle functor of $f$ in the complex analytic topology.
Hence
the Euler characteristic $\chi(\FF_0(\hat{f}_{E}))$ is the same as the  Euler characteristic of the Milnor fiber of $f$. From  \cite[\S 1.2]{Be}, the Behrend function at the origin is given by the Euler characteristic of the analytic Milnor fiber in the formula.

We give another proof without using the result of Behrend and Getzler.  Here we use the result of Joyce etc in \cite{Joyce} and \cite{BBBJ} that the moduli stack $\MM$ locally at the point $E\in\MM$ is a form of the quotient stack $[\Crit(g)/G]$, where $g$ is a regular function on a smooth affine scheme 
$U$ with dimension $\dim(\Ext^1(E,E))$, and $G$ is an algebraic reductive group.  Here we can take $U$ as an open neighbourhood of $\Ext^1(E,E)$.  Taking the formal completion of $g$ along the origin we get a formal scheme
$$\hat{g}: \UU\to\spf(R).$$
The germ moduli scheme $\widehat{\MM}$ at the point $E$ is given by 
$\Crit(\hat{g})\subset \UU$. 

On the other hand the germ moduli scheme is given by the critical locus of the  superpotential function $\hat{f}_{E}$, i.e.
$\Crit(\hat{f}_{E})\subset \widehat{\MM}$, since the formal deformation of the point $E$ is controlled by the cyclic $L_\infty$-algebra
$\Ext^*(E,E)$.  These two germ moduli schemes are isomorphic, since they represent the same germ moduli scheme $\widehat{\MM}_{E}$.  Hence the two formal schemes $(\XX, \hat{f}_{E})$ and $(\UU, \hat{g})$ are isomorphic.  From \cite[Proposition 8.8]{Ni}, the analytic Milnor fibers $\FF_{0}(\hat{f}_{E})$ and $\FF_0(\hat{g})$ are isomorphic over $\cc(\!(t)\!)$ and their Euler characteristics are the same.  By the comparison theorem in \cite{Ber4}, the Euler characteristic of $\FF_0(\hat{g})$ is the same as the Euler characteristic of the Milnor fiber of $g$, which gives the value of the Behrend function in the formula of the proposition. 
\end{proof}

\begin{lem}
If $E_1$ and $E_2$ are semi-Schur, then $E_1\oplus E_2$ is also semi-Schur.
\end{lem}
\begin{pf}
This is clear since both $E_1$ and $E_2$ are in the abelian category $\mathcal{A}$.
\end{pf}

\sss
\begin{thm}\label{main:Joyce-Song:formula}
The following two Joyce-Song formulas hold for semi-Schur objects.
\begin{enumerate}
\item $$\nu_{\MM}(E_1\oplus E_2)=(-1)^{\chi(E_1,E_2)}\nu_{\MM}(E_1)\nu_{\MM}(E_2)\,.$$
Here, $\chi(E_1,E_2)=\sum_i(-1)^i\dim \Ext^i(E_1,E_2)$ is the Euler form.
\item 
$$\int_{F\in \mathbb{P}(\ext^{1}(E_2,E_1))}\nu_{\MM}(F)d\chi-\int_{F\in \mathbb{P}(\ext^{1}(E_1,E_2))}\nu_{\MM}(F)d\chi$$
$$=(\dim(\Ext^{1}(E_2,E_1))-\dim(\Ext^{1}(E_1,E_2)))\nu_{\MM}(E_1\oplus E_2).$$
\end{enumerate}
\end{thm}

%%%------------------------------------
\sss \textbf{Proof of formula (1) in Theorem \ref{main:Joyce-Song:formula}:}

Let $E=E_1\oplus E_2$ be an derived semi-Schur object in the derived category of coherent sheaves on the Calabi-Yau threefold $Y$.

Let $R\Hom(E,E)$, $R\Hom(E_1,E_1)$ and $R\Hom(E_2,E_2)$ be the corresponding dg Lie algebras. Then on the cohomology $\Ext(E,E)$, $\Ext(E_1,E_1)$ and $\Ext(E_2,E_2)$ there exist cyclic $L_\infty$-algebras, which we denote them by $L_E$, $L_{E_1}$ and $L_{E_2}$.

Let $X_E:=\Ext^{1}(E,E)$, $X_{E_1}:=\Ext^{1}(E_1,E_1)$ and $X_{E_2}:=\Ext^{1}(E_2,E_2)$.  
Let $X_{E_1,E_2}:=\Ext^{1}(E_1,E_2)$ and $X_{E_2,E_1}:=\Ext^{1}(E_2,E_1)$.
Then from the definition of $\Ext^{1}(E,E)$, 
$$X_{E}=X_{E_1}\oplus X_{E_2}\oplus X_{E_1,E_2}\oplus X_{E_2,E_1}.$$
Let $S=\{x,y\}, T=\{z, w\}$ be the coordinates of $X_{E_1}$ and $X_{E_2}$;  $X_{E_1,E_2}$ and $X_{E_2,E_1}$ respectively.  From the definition of superpotential function determined by the higher products of the cyclic $L_\infty$-algebra, the potential function $f_{E}$ has the following decomposition:
$$f_E=f_{E_1}+f_{E_2}+(\mbox{mixed terms of} ~x, y ,z,w),$$
where $f_{E_1}(x)$ and $f_{E_2}(y)$ are the potential functions for the $L_\infty$-algebras $L_{E_1}$ and 
$L_{E_2}$.
This superpotential function $f_{E}$ belongs to the $R$-algebra $R\{T\}[[S]]$, since when putting $S=0$, then $f=0\in R\{T\}$. 

\begin{prop}
We have
$$(1-\chi(E))=(1-\chi(E_1))\cdot (1-\chi(E_2)).$$
\end{prop}
\begin{pf}
Let $P\in X_{E}$, $P_1\in X_{E_1}$ and $P_2\in X_{E_2}$ be the origins. 
We have 
$$\chi(E)=\chi(\FF_{f_{E}}(P))$$ 
$$\chi(E_1)=\chi(\FF_{f_{E_{1}}}(P_1)),$$
and
$$\chi(E_2)=\chi(\FF_{f_{E_{2}}}(P_2)),$$
where $\FF_{f_E}(P), \FF_{f_{E_{1}}}(P_1), \FF_{f_{E_{2}}}(P_2)$ be the analytic Milnor fibers of the associated superpotential function.

Let $T:=\{id_{E_1}+\gamma id_{E_2}: \gamma\in U(1)\}$
acts on $\Ext^{1}(E,E)$ by 
$$\gamma(s)=\gamma\circ s\circ\gamma^{-1}$$
for $\gamma\in T$. Then the fixed points are 
$$\Ext^1(E,E)^{T}=\Ext^1(E_1,E_1)\oplus\Ext^1(E_2,E_2).$$
It is easy to see that 
 $\chi(\FF_{f_{E}}(P))=\chi(\FF_{f_{E}}(P)^{T})$. Let $P^{T}=(P_1,P_2)\in \Ext^1(E_1,E_1)\oplus\Ext^1(E_2,E_2)$
 be the origin. The potential function $f_{E}|_{\ext^1(E_1,E_1)\oplus\ext^1(E_2,E_2)}=f_{E_1}+f_{E_2}$.  Then 
 \begin{align*}
(1- \chi(E))&=\left(1-\chi(\FF_{f_{E}}(P))\right)\\
 &=\left(1-\chi(\FF_{f_{E}|_{\ext^1(E_1,E_1)\oplus\ext^1(E_2,E_2)}}(P))\right) \\
 &=\left(1- \chi(\FF_{f_{E_{1}}}(P_1))\right)\cdot
\left(1-\chi(\FF_{f_{E_{2}}}(P_2))\right) \\
&=(1-\chi(E_1))\cdot(1-\chi(E_2)).
 \end{align*}
\end{pf}

From the definition of Behrend functions 
$$
\begin{array}{l}
\nu_{\MM}(E)=(-1)^{\dim(\ext^0(E,E))+\dim(\ext^1(E,E))}(1-\chi(L_{E})) \\
\nu_{\MM}(E_1)=(-1)^{\dim(\ext^0(E_1,E_1))+\dim(\ext^1(E_1,E_1))}(1-\chi(L_{E_{1}})) \\
\nu_{\MM}(E_2)=(-1)^{\dim(\ext^0(E_2,E_2))+\dim(\ext^1(E_2,E_2))}(1-\chi(L_{E_{2}})) 
\end{array}
$$

Then we compute
\begin{align*}
\nu_{\MM}(E)&=(-1)^{\dim(\ext^0(E,E))+\dim(\ext^1(E,E))}(1-\chi(L_{E}))\\
&=(-1)^{\dim(\ext^0(E_1,E_1))+\dim(\ext^0(E_2,E_2))+\dim(\ext^0(E_1,E_2))+\dim(\ext^0(E_2,E_1))}\cdot \\
&\quad (-1)^{\dim(\ext^1(E_1,E_1))+\dim(\ext^1(E_2,E_2))+\dim(\ext^1(E_1,E_2))+\dim(\ext^1(E_2,E_1))}\\
&\cdot(1-\chi(L_{E_1}))\cdot(1-\chi(L_{E_2})) \\
&=(-1)^{\dim(\ext^0(E_1,E_2))+\dim(\ext^0(E_2,E_1))+\ext^1(E_1,E_2))+\dim(\ext^1(E_2,E_1))}\cdot \\
&\quad (-1)^{\dim(\ext^0(E_1,E_1))+\dim(\ext^1(E_1,E_1))}(1-\chi(L_{E_1}))\cdot \\
&(-1)^{\dim(\ext^0(E_2,E_2))+\ext^1(E_2,E_2))}(1-\chi(L_{E_2})) \\
&=(-1)^{\dim(\ext^0(E_1,E_2))+\dim(\ext^0(E_2,E_1))+\ext^1(E_1,E_2))+\dim(\ext^1(E_2,E_1))} \cdot \nu_{\MM}(E_1)\cdot\nu_{\MM}(E_2) \\
&=(-1)^{\chi(E_1,E_2)}\cdot \nu_{\MM}(E_1)\cdot\nu_{\MM}(E_2).
\end{align*}
This completes the proof of formula (1).
%%%------------------------------------
\sss \textbf{Proof of formula (2) in Theorem \ref{main:Joyce-Song:formula}:}

Let 
$$\Ext^1(E,E)=\Ext^1(E_1,E_1)\oplus\Ext^1(E_2,E_2)\oplus\Ext^1(E_1,E_2)\oplus\Ext^1(E_2,E_1).$$
Then 
$[\epsilon_{21}]\in \pp(\Ext^1(E_2,E_1))$ represents an element 
$(0,0,0,\epsilon_{21})$ in $\Ext^1(E,E)$, and  $[\epsilon_{12}]\in \pp(\Ext^1(E_1,E_2))$ represents an element 
$(0,0,\epsilon_{12},0)$ in $\Ext^1(E,E)$.  Then the formula (2) is equivalent to the following formula
$$
\int_{F\in \mathbb{P}(\mbox{\ext}^{1}(E_2,E_1))}(1-\chi(\FF_{f}(0,0,0,\epsilon_{21}))d\chi-\int_{F\in \mathbb{P}(\ext^{1}(E_1,E_2))}(1-\chi(\FF_{f}(0,0,\epsilon_{12},0))d\chi$$
$$=(\dim(\Ext^{1}(E_2,E_1))-\dim(\Ext^{1}(E_1,E_2)))(1-\chi(\FF_{f|_{\ext^1(E_1,E_1)\oplus\ext^1(E_2,E_2)}}(0)).$$
In turn it is equivalent to the following formula:
\begin{align}\label{formula2:equivalence:formula}
&\int_{F\in \mathbb{P}(\mbox{\ext}^{1}(E_2,E_1))}\chi(\FF_{f}(0,0,0,\epsilon_{21})d\chi-\int_{F\in \mathbb{P}(\ext^{1}(E_1,E_2))}\chi(\FF_{f}(0,0,\epsilon_{12},0)d\chi  \\
&=(\dim(\Ext^{1}(E_2,E_1))-\dim(\Ext^{1}(E_1,E_2)))\cdot \chi(\FF_{f|_{\ext^1(E_1,E_1)\oplus\ext^1(E_2,E_2)}}(0). \nonumber 
\end{align}

We prove the formula in (\ref{formula2:equivalence:formula}).
Let 
$$U:=\{(\epsilon_{11},\epsilon_{22},\epsilon_{12},\epsilon_{21})\in \Ext^1(E,E): \epsilon_{21}\neq 0\}$$
and 
$$V:=\{(\epsilon_{11},\epsilon_{22},\epsilon_{12},\epsilon_{21})\in U: \epsilon_{12}= 0\}.$$
We prove the following formulas 
\begin{align}\label{formula-1}
&\chi(\FF_{f}(0,0,0,\epsilon_{21})= \\
& \int_{[\epsilon_{12}]\in \pp(\ext^1(E_1,E_2))}\chi(\FF_{\tilde{f}}(0,0,[\epsilon_{12}],0,\epsilon_{21}))d\chi+(1-\dim(\Ext^1(E_1,E_2)))\chi(\FF_{f|_{V}}(0,0,0,\epsilon_{21})). \nonumber
\end{align}

and 
\begin{align}\label{formula-2}
\chi(\FF_{f|_{V}}(0,0,0,\epsilon_{21}))=\chi(\FF_{f|_{\ext^1(E_1,E_1)\oplus\ext^1(E_2,E_2)}}(0,0,0,0)).
\end{align}

The formula (\ref{formula-1}) is from Lemma \ref{lemma-key}, by applying the  formula to the formal blow-up 
of the formal $R$-scheme 
$$\hat{f}: \XX|_{U}:=\widehat{\Ext^1(E,E)|_{U}}\to\spf(R)$$
along the formal subscheme $\widehat{V}\subset \XX|_{U}$. 
Then evaluating the formula at the point $v\in \Crit(f)$, we get the formula (\ref{formula-1}).

To prove formula (\ref{formula-2}),  note that 
$$V:=\{(\epsilon_{11},\epsilon_{22},\epsilon_{12},\epsilon_{21})\in 
\Ext^1(E,E): \epsilon_{12}=0\}.$$
So we have 
$$V=\Ext^1(E_1,E_1)\oplus\Ext^1(E_2,E_2)\oplus W,$$ 
where
$W=\{(\epsilon_{11},\epsilon_{22},\epsilon_{12},\epsilon_{21})\in \Ext^1(E,E):
\epsilon_{11}=\epsilon_{22}=\epsilon_{12}= 0,\epsilon_{21}\neq 0\}$.
From the potential function $f_{E}: \Ext^1(E,E)\to\cc$, the restriction $f|_{V}$ is given by
$f|_{V}=f_{E_1}+f_{E_2}+0$. 
Hence from the Thom-Sebastiani theorem for Euler characteristic of cyclic L-infinity algebras
we have 
\begin{align*}
\chi(\FF_{f|_{V}}((0,0,0,\epsilon_{21})))
&=\chi(\FF_{f|_{\ext^1(E_1,E_1)\oplus\ext^1(E_2,E_2)}}((0,0,0,0)))\cdot\chi(\FF_{0}((0,0,0,\epsilon_{21}))) \\
&=\chi(\FF_{f|_{\ext^1(E_1,E_1)\oplus\ext^1(E_2,E_2)}}((0,0,0,0))),
\end{align*}
since $\chi(\FF_{0}((0,0,0,\epsilon_{21})))=1$.

Similarly let
$$U^\prime:=\{(\epsilon_{11},\epsilon_{22},\epsilon_{12},\epsilon_{21})\in \Ext^1(E,E): \epsilon_{12}\neq 0\}$$
and 
$$V^\prime:=\{(\epsilon_{11},\epsilon_{22},\epsilon_{12},\epsilon_{21})\in U^\prime: \epsilon_{21}= 0\}.$$
Then we have
\begin{align}\label{formula-11}
&\chi(\FF_{f}(0,0,\epsilon_{12},0)= \\
&\int_{[\epsilon_{21}]\in \pp(\ext^1(E_2,E_1))}\chi(\FF_{\tilde{f}}(0,0,\epsilon_{12},0,[\epsilon_{21}]))d\chi+(1-\dim(\Ext^1(E_2,E_1)))\chi(\FF_{f|_{V}}(0,0,\epsilon_{12},0)). \nonumber
\end{align}
and
\begin{align}\label{formula-22}
\chi(\FF_{f|_{V}}(0,0,\epsilon_{12},0))=\chi(\FF_{f|_{\ext^1(E_1,E_1)\oplus\ext^1(E_2,E_2)}}(0,0,0,0)).
\end{align}

From formulas (\ref{formula-1}) and  (\ref{formula-11}),  taking integral yields
\begin{align}\label{formula-111}
&\int_{\epsilon_{21}\in \pp(\ext^1(E_2,E_1))}\chi(\FF_{f}(0,0,0,\epsilon_{21})d\chi=\\ 
&\int_{[\epsilon_{12}]\in \pp(\ext^1(E_1,E_2)),
[\epsilon_{21}]\in \pp(\ext^1(E_2,E_1))}\chi(\FF_{\tilde{f}}(0,0,[\epsilon_{12}],0,[\epsilon_{21}]))d\chi \nonumber \\
&+\dim(\Ext^1(E_2,E_1))(1-\dim(\Ext^1(E_1,E_2)))\chi(\FF_{f|_{V}}(0,0,0,\epsilon_{21})). \nonumber
\end{align}
and 
\begin{align}\label{formula-222}
&\int_{\epsilon_{12}\in \pp(\ext^1(E_1,E_2))}\chi(\FF_{f}(0,0,\epsilon_{21},0)d\chi=\\ 
&\int_{[\epsilon_{21}]\in \pp(\ext^1(E_2,E_1)),
[\epsilon_{12}]\in \pp(\ext^1(E_1,E_2))}\chi(\FF_{\tilde{f}}(0,0,[\epsilon_{12}],0,[\epsilon_{21}]))d\chi \nonumber \\
&+\dim(\Ext^1(E_1,E_2))(1-\dim(\Ext^1(E_2,E_1)))\chi(\FF_{f|_{V}}(0,0,\epsilon_{12},0)). \nonumber
\end{align}
Then if we   minus formula (\ref{formula-222}) from  formula (\ref{formula-111}), the formula 
(\ref{formula2:equivalence:formula}) is proved.

\sss \textbf{Application to stable pairs:}

Let $Y$ be a Calabi-Yau threefold.  A stable pair is given by data $I=[\mathcal{O}_Y\stackrel{s}{\rightarrow} F]$, where $F$ is a pure dimension one coherent sheaf over $Y$ and the cokernel of $s$ is supported in dimension zero.  Fix data $(\beta,n)$ as in \cite{PT}, denote by $P_n(Y,\beta)$ the moduli scheme of stable pairs, the Behrend function associate to it is denoted by  $\nu_P$. The coherent sheaf $F$ may not be stable except that $\beta$ is irreducible, but there is a moduli stack $\MM$ of them. We denote $\nu_{\MM}$ the corresponding Behrend function on it. 

\begin{prop}
$$\nu_{P}=(-1)^{\chi(F)}\nu_{\MM}.$$
\end{prop}
\begin{pf}
There is a distinguished triangle $F[-1]\to I\to \mathcal{O}_{Y}\to F$.
Consider the direct sum object $F[-1]\oplus I$, Formula (1) gives
$$\nu_{\MM}(F[-1]\oplus \mathcal{O}_Y)=(-1)^{\chi(F)-1}\nu_{\MM}(F[-1])\cdot\nu_{\MM}(\mathcal{O}_{Y})=(-1)^{\chi(F)}\nu_{\MM}(F),$$
where the last equality is from the fact that  $\chi(F[-1],\mathcal{O}_{Y})=(-1)^{\chi(F)-1}$, $\nu_{\MM}(\mathcal{O}_{Y})=1$ since $\mathcal{O}_{Y}$ is a spherical object and $\nu_{\MM}(F[-1])=(-1)\nu_{\MM}(F)$.

Fix a stable pair $I$,  up to isomorphism there is a unique nontrivial triangle $F[-1]\to I\to \mathcal{O}_{Y}\to F$ and hence Formula (2) gives 
$$\nu_{P}(I)=\nu_{\MM}(F[-1]\oplus \mathcal{O}_Y)=(-1)^{\chi(F)}\nu_{\MM}(F).$$
\end{pf}

\begin{rmk}
This result is proved using the property of the moduli space of stable pairs by Bridgeland \cite{Bridgeland}, motivated by the unpublished notes of Pandharipande-Thomas.
\end{rmk}
%%%----------------------------------------------------------------------
\section{The Motivic Thom-Sebastiani formula.}\label{motivic}

\subsection*{Resolution of singularities.}

\sss
Let $(L,\mu, \ae)$ be a cyclic $L_\infty$-algebra of dimension $3$.
The cyclicity property on the cohomology $H(L)$ gives a formal potential function
\begin{equation}
f: H^{1}(L)\longrightarrow \mathbb{C}
\end{equation}
defined by 
$$z\mapsto 
f(z)=\sum_{n=2}^\infty\frac{(-1)^{\frac{n(n+1)}{2}}}{(n+1)!}\kappa\big(\nu_n(z,\ldots,z),z\big).
$$

\sss In general $f$ is a formal series and $f(0)=0$, where $0\in H^{1}(L)$ is the origin. 
For simplicity, let $X:=H^{1}(L)$ and let $m=\dim(X)$.
Let 
$$\hat{f}: \XX\longrightarrow \spf(R)$$
be the $t$-adic formal completion of $f$. Then 
$\XX$ is a special formal $R$-scheme in sense of \cite{Ber2} and \cite{Ni}. 

\sss From \cite{Tem} and \cite{Jiang},   let 
\begin{equation}
h:  \YY\longrightarrow \XX
\end{equation}
be the  resolution  of singularities of formal  $R$-scheme $\XX$.

Let $E_i$, $i\in A$, be the set of irreducible components of the exceptional divisors of the resolution. 
For $I\subset A$,  we set 
$$E_{I}:=\bigcap_{i\in I}E_{i}$$
and 
$$E_{I}^{\circ}:=E_{I}\setminus \bigcup_{j\notin I}E_j.$$
Let $m_{i}$ be the multiplicity of the component $E_i$, which means that 
the special fiber of the resolution is 
$$\sum_{i\in A}m_iE_i.$$

\subsection*{Grothendieck group of varieties.}

\sss In this section we briefly review the Grothendieck group of varieties. 
Let $S$ be an algebraic variety over $\cc$. Let $\Var_{S}$ be the category of 
$S$-varieties.

\sss Let $K_0(\Var_{S})$ be the Grothendieck group of $S$-varieties.  By definition $K_0(\Var_{S})$ 
is an abelian group with generators given by all the varieties $[X]$'s, where $X\rightarrow S$ are $S$-varieties,  and the relations are $[X]=[Y]$, if $X$ is isomorphic to $Y$, and 
$[X]=[Y]+[X\setminus Y]$ if $Y$ is a Zariski closed subvariety of $X$.
Let $[X],  [Y]\in K_0(\Var_{S})$,  and define $[X][Y]=[X\times_{S} Y]$.  Then 
we have a product on $K_0(\Var_{S})$. 
Let $\mathbb{L}$ represent the class of $[\mathbb{A}_{\cc}^{1}\times S]$.
Let $\mathcal{M}_{S}=K_0(\Var_{S})[\mathbb{L}^{-1}]$
be the ring by inverting the class $\mathbb{L}$ in the ring $K_0(\Var_{S})$.

If $S$ is a point $\spec (\cc)$, we write $K_0(\Var_{\cc})$ for the Grothendieck ring of $\cc$-varieties.
One can take the map $\Var_{\cc}\longrightarrow K_0(\Var_{\cc})$ to be the universal Euler characteristic.
After inverting the class $\mathbb{L}=[\mathbb{A}_{\cc}^{1}]$, we get the ring $\mathcal{M}_{\cc}$.

\sss We introduce the equivariant Grothendieck group defined in \cite{DeLo1}.
Let $\mu_n$ be the cyclic group of order $n$, which can be taken as the algebraic variety
$Spec (\cc[x]/(x^n-1))$. Let $\mu_{md}\longrightarrow \mu_{n}$ be the map $x\mapsto x^{d}$. Then 
all the groups $\mu_{n}$ form a projective system. Let 
$$\underleftarrow{lim}_{n}\mu_{n}$$
be the direct limit.

Suppose that $X$ is a $S$-variety. The action $\mu_{n}\times X\longrightarrow X$ is called a $good$ 
action if  each orbit is contained in an affine subvariety of $X$.  A good $\hat{\mu}$-action on $X$ is an action of $\hat{\mu}$ which factors through a good $\mu_n$-action for some $n$.

The $equivariant ~Grothendieck~ group$ $K^{\hat{\mu}}_0(\Var_{S})$ is defined as follows:
The generators are $S$-varieties $[X]$ with a good $\hat{\mu}$-action; and the relations are:
$[X,\hat{\mu}]=[Y,\hat{\mu}]$ if $X$ is isomorphic to $Y$ as $\hat{\mu}$-equivariant $S$-varieties,  
and $[X,\hat{\mu}]=[Y,\hat{\mu}]+[X\setminus Y, \hat{\mu}]$ if $Y$ is a Zariski closed subvariety
of $X$ with the $\hat{\mu}$-action induced from that on $X$,  if $V$ is an affine variety with a good 
$\hat{\mu}$-action, then $[X\times V,\hat{\mu}]=[X\times \mathbb{A}_{\cc}^{n},\hat{\mu}]$.  The group 
$K^{\hat{\mu}}_0(\Var_{S})$ has a ring structure if we define the product as the fibre product with the good $\hat{\mu}$-action.  Still we let $\mathbb{L}$  represent the class $[S\times \mathbb{A}_{\cc}^{1},\hat{\mu}]$ and let $\mathcal{M}_{S}^{\hat{\mu}}=K^{\hat{\mu}}_0(\Var_{S})[\mathbb{L}^{-1}]$ be the ring obtained from $K^{\hat{\mu}}_0(\Var_{S})$ by inverting the class $\mathbb{L}$.

If $S=\spec(\cc)$, then we write $K^{\hat{\mu}}_0(\Var_{S})$ as $K^{\hat{\mu}}_0(\Var_{\cc})$, and $\mathcal{M}_{S}^{\hat{\mu}}$ as $\mathcal{M}_{\cc}^{\hat{\mu}}$.  Let  $s\in S$ be a geometric point. Then we have natural maps $K^{\hat{\mu}}_0(\Var_{S})\longrightarrow K^{\hat{\mu}}_0(\Var_{\cc})$ and $\mathcal{M}_{S}^{\hat{\mu}}\longrightarrow \mathcal{M}_{\cc}^{\hat{\mu}}$ given by the correspondence
$[X,\hat{\mu}]\mapsto [X_s,\hat{\mu}]$.

\subsection*{The motivic Milnor fiber.}

\sss Let $(L,\mu,\ae)$ be a cyclic $L_\infty$-algebra of dimension $3$.  Then we have a cyclic $L_\infty$-algebra structure $(H(L),\nu,\kappa)$ on the cohomology $H(L)$.  On $H^{1}(L)$ we have a formal series 
$$f: H^{1}(L)\longrightarrow \mathbb{C}.$$
Recall  that 
$$h: \YY\longrightarrow \XX$$
is the resolution of singularities of the special formal $R$-scheme 
$\XX$.

\sss Let $m_{I}=gcd(m_i)_{i\in I}$. Let $U$ be an affine Zariski open subset of $\YY$, such that, 
on $U$, $f\circ h=uv^{m_{I}}$, with $u$ a unit in $U$ and $v$ a morphism from 
$U$ to $\mathbb{A}_{\cc}^{1}$. The restriction of $E_{I}^{\circ}\cap U$, which we denote by
$\tilde{E}_{I}^{\circ}\cap U$, is defined by
$$\lbrace{(z,y)\in \mathbb{A}_{\cc}^{1}\times (E_{I}^{\circ}\cap U)| z^{m_{I}}=u^{-1}\rbrace}.$$
The $E_{I}^{\circ}$ can be covered by the open subsets $U$ of $Y$.  We can glue together all such 
constructions and get the Galois cover
$$\tilde{E}_{I}^{\circ}\longrightarrow E_{I}^{\circ}$$
with Galois group $\mu_{m_{I}}$.
Remember that $\hat{\mu}=\underleftarrow{lim} \mu_{n}$ is the direct limit of the groups
$\mu_{n}$. Then there is a natural $\hat{\mu}$ action on $\tilde{E}_{I}^{\circ}$.
Thus we get 
$[\tilde{E}_{I}^{\circ}]\in \mathcal{M}_{X_0}^{\hat{\mu}}$. The following definition is given in \cite{Jiang}. 

\begin{defn}
The motivic Milnor fiber of the cyclic $L_\infty$-algebra $L$ is defined as follows:
$$\mathcal{S}_{f,0}(L):=MF_0(L)=\sum_{\emptyset\neq I\subset A}(1-\mathbb{L})^{|I|-1}[\tilde{E}^{\circ}_{I}\cap h^{-1}(0)].$$
\end{defn}
It is clear that $\mathcal{S}_{f,0}(L)\in \mathcal{M}_{\cc}^{\hat{\mu}}.$

\subsection*{The Thom-Sebastiani Formula.}

\sss Let $L_1$ and $L_2$ be two cyclic $L_\infty$-algebras of dimension three. 
From Section \ref{Thom}, there exists a cyclic $L_\infty$-algebra structure on the direct sum
$L_1\oplus L_2=L$. On the cohomology 
$H^1(L)=H^1(L_1)\oplus H^1(L_2)$, the superpotential function 
$$f: H^1(L)\to \cc$$
has a split $f=f_1+f_2$, where $f_i$ is the superpotential function on $H^1(L_i)$ for $i=1,2$.  
The motivic Milnor fiber of $L$ is defined in a similar way.  The Thom-Sebastiani Theorem is stated as follows:

\begin{thm}(\cite{KS2}, \cite{DeLo1}, \cite{Thuong})
$$(1-\sS_{f,(0,0)}(L_1\oplus L_2))=(1-\sS_{f_1,0}(L_1))\cdot (1-\sS_{f_2,0}(L_2)).$$
\end{thm}
\begin{rmk}
The motivic Thom-Sebastiani Theorem for regular functions is proved in \cite{DeLo1}.
In \cite{Thuong}, Le proves the formal function version of the motivic Thom-Sebastiani Theorem. 
\end{rmk}

\sss \textbf{Conjecture on the Joyce-Song formula on motivic Milnor fibers:}

Let $E$ be a semi-Schur object in the derived category, define 
$$\sS_{0}(E)=\sS_{f,0}(L_{E}),$$
where the cyclic $L_\infty$-algebra $L_E$ is the $L_\infty$-algebra
$\Ext^*(E,E)$.  Let $f: \Ext^1(E,E)\to\cc$ be the formal potential function and 
$$\hat{f}: \XX\to\spf(R)$$
the corresponding special formal $R$-scheme. Recall that 
$$h: \YY\to\XX$$ is the resolution of singularities.  If we have a formal subscheme $\ZZ\subset \XX$, then we define $\sS_{\ZZ}(\hat{f})$ to be the motivic Milnor fiber of 
$\ZZ$: 
$$\sS_{\ZZ}(\hat{f}):=\sum_{\emptyset\neq I\subset A}(1-\ll)^{|I|-1}[\tilde{E}^{\circ}_{I}\cap h^{-1}(\ZZ)].$$

We give the motivic version of Joyce-Song formulas.
\begin{conjecture}\label{Con_Joyce}
\begin{enumerate}
\item $$(1-\mathcal{S}_{((0,0))}(E_1\oplus E_2))=(1-\sS_{0}(E_1))\cdot (1-\sS_{0}(E_2))\,.$$
\item 
\begin{align*}
&\int_{F\in\mathbb{P}(\ext^{1}(E_2,E_1))}(1-\sS_{0}(F))-\int_{F\in\mathbb{P}(\ext^{1}(E_1,E_2))}(1-\sS_{0}(F))\\
&=([\pp^{\dim\ext^{1}(E_2,E_1)}]-[\pp^{\dim\ext^{1}(E_1,E_2)}])\left(1-\sS_{f|_{X_{E_1}\oplus X_{E_2}},0}\right).
\end{align*}
\end{enumerate}
where $\int_{\XX_0}(-): \mM^{\hat{\mu}}_{\XX_0}\to \mM^{\hat{\mu}}_{\cc}$ is the pushforward of motivic vanishing cycles. 
\end{conjecture}

\sss We give a little  explanation about the conjecture.  For any $E\in\Coh(Y)$, $\sS_{0}(E)$ is the motivic Milnor fiber of $E$, and $(1-\sS_0(E))$ is the analogue of motivic vanishing cycle.  Let $E:=E_1\oplus E_2$.  
Then 
$$\Ext^1(E,E)=\Ext^1(E_1,E_1)\oplus \Ext^1(E_2,E_2)\oplus \Ext^1(E_1,E_2)\oplus \Ext^1(E_2,E_1).$$
Let
$$\phi: \widetilde{\XX}\to \XX$$
be the formal blow-up of $\XX$ along
the completion $\YY\subset \XX$, where 
$\YY=\widehat{Y}$ and $Y:=\Ext^1(E_1,E_1)\oplus \Ext^1(E_1,E_1)\oplus 0\oplus\Ext^1(E_2,E_1)\subset \Ext^1(E,E)$.
We denote by $\ZZ:=\widehat{\Ext^1(E_1, E_2)}\subset \XX$. 
Since the motivic vanishing cycle is constructible,   then
the integration 
$$\int_{F\in\mathbb{P}(\ext^{1}(E_1,E_2))}(1-\sS_{0}(F))$$
can be understood as the motivic cycle
$\sS_{\ZZ}(\widetilde{\hat{f}})$, where 
$$\widetilde{\hat{f}}=\phi\circ \hat{f}: \widetilde{\XX}\to\spf(R)$$
is the composition of $\phi$ and $\hat{f}$. 
The meaning of the  integration
$$\int_{F\in\mathbb{P}(\ext^{1}(E_2,E_1))}(1-\sS_{0}(F))$$
is similar.

\begin{rmk}
These two conjectural formulas  are related  to Conjecture 4.2 in the paper \cite{KS2} by Kontsevich and Soibelman. Note that over field of  characteristic $p$, this conjecture was proved in \cite{KS2}.  Recently this conjecture is proved by Le in \cite{Thuong2}. 

The author strongly believes that Conjecture \ref{Con_Joyce} is true for field of characteristic zero. These two formulas are the crucial fact for  the wall-crossing of  motivic Donaldson-Thomas invariants in \cite{JS},  \cite{KS2}, since it will imply that the homomorphism from motivic Hall algebra to the motivic quantum torus is a Poisson homomorphism. 
In \cite{Jiang2}, we will address these conjectures.  
\end{rmk}

%%%----------------------------------------------------------------------
%%%----------------------------------------------------------------------

\subsection*{}

% ------------------------------------------------------------------------
\end{document}